%% file: main.tex
\documentclass{amsart}
\usepackage{hyperref}
\usepackage[utf8]{inputenc}
\usepackage[color,arrow,all]{xy}
\usepackage{tikz,keyval,enumerate}
\usepackage{lineno}
\usepackage{forest}
\usepackage[new]{old-arrows}
\usetikzlibrary{arrows,positioning,calc,fit,graphs}
\tikzset{%
pics/bullet/.style args={#1,#2}{
    code={
    \node at (0,1) [draw,circle,fill,minimum size=2mm,inner sep=0pt,label={$#1$}](#1){};
    \node at (2,1) [draw,circle,fill,minimum size=2mm,inner sep=0pt,,label={$#2$}](#2){};
    \node at (1,0) [draw,circle,fill,minimum size=2mm,inner sep=0pt](bb){};
    \draw (bb)--(#1) (bb)--(#2);
    }}}
\usepackage{amsmath,amssymb,amsthm,amsfonts}
\usepackage{comment}
\usepackage{xcolor}
\usepackage{listings,soul}

\lstdefinestyle{mystyle}{
    backgroundcolor=\color{backcolour},   
    commentstyle=\color{codegreen},
    keywordstyle=\color{magenta},
    numberstyle=\tiny\color{codegray},
    stringstyle=\color{codepurple},
    basicstyle={
    \fontencoding{T1}\footnotesize\fontfamily{phv}\fontseries{m}\selectfont},
    breakatwhitespace=false,         
    breaklines=true,                 
    captionpos=b,                    
    keepspaces=false,                 
    numbers=left,                    
    numbersep=5pt,                  
    showspaces=false,                
    showstringspaces=false,
    showtabs=false,                  
    tabsize=2
}

\hypersetup{
    colorlinks=true,
    citecolor=brown,
    filecolor=brown,
    linkcolor=magenta,
    urlcolor=magenta
}

\definecolor{codegreen}{rgb}{0,0.6,0}
\definecolor{codegray}{rgb}{0.5,0.5,0.5}
\definecolor{codepurple}{rgb}{0.58,0,0.82}
\definecolor{backcolour}{rgb}{0.95,0.95,0.92}
 
\theoremstyle{plain}
\newtheorem{theorem}{Theorem}[section] 
\newtheorem{proposition}[theorem]{Proposition}
\newtheorem{lemma}[theorem]{Lemma}
\newtheorem{conj}[theorem]{Conjecture}

\theoremstyle{remark}

\newtheorem{remark}[theorem]{Remark} 

\newtheorem{example}[theorem]{Example}

\def\A{\mathsf{A}}
\def\D{\mathsf{D}}
\def\lra{\longrightarrow}

\newcommand{\Ind}{\mathsf{Ind}}
\newcommand{\Res}{\mathsf{Res}}
\newcommand{\Aut}{\mathsf{Aut}}
\newcommand{\End}{\mathsf{End}}
\newcommand{\Id}{\mathsf{Id}}
\newcommand{\id}{\mathsf{id}}
\newcommand{\Hom}{\mathsf{Hom}}

\newcommand{\Mod}{\mathsf{mod}}
\newcommand{\bimod}{\mathsf{bimod}}

\newcommand{\GA}[1]{\mathbb{C}\A_{#1}}
\newcommand{\GZ}[1]{\mathbb{C}Z_{#1}}

\newcommand{\orbit}{\mathcal{O}}
\newcommand{\orbitsum}{\mathfrak{o}}

\makeatletter
\newcommand{\vast}{\bBigg@{4.15}}
\newcommand{\Vast}{\bBigg@{6}}
\makeatother

\allowdisplaybreaks
\begin{document}

\title[Natural transformations on iterated symmetric group products]{Natural transformations between induction and restriction on iterated wreath product of symmetric group of order $2$}
\author{Mee Seong Im}
\address{Department of Mathematical Sciences, United States Military Academy, West Point, NY 10996, USA}
\curraddr{Department of Mathematics, United States Naval Academy, Annapolis, MD 21402, USA}
\email[corresponding author]{meeseongim@gmail.com}

\author{Can Ozan O\u{g}uz}
\address{Institute of Information Technologies, Gebze Technical University, \.{I}stanbul, Turkey}
\email{canozanoguz@gmail.com} 
\date{October 9, 2022}

\makeatletter
\@namedef{subjclassname@2020}{%
  \textup{2020} Mathematics Subject Classification}
  
\makeatother


\providecommand{\keywords}[1]{\textbf{\textit{Key words and phrases.}} #1}

\begin{abstract}  
    Let $\mathbb{C}\A_n = \mathbb{C}[S_2\wr S_2 \wr\cdots \wr S_2]$ be the group algebra of $n$-step iterated wreath product. We prove some structural properties of $\A_n$ such as their centers, centralizers, right and double cosets. We apply these results to explicitly write down Mackey theorem for groups $\A_n$ and give a partial description of  the natural transformations between induction and restriction functors on the representations of the iterated wreath product tower by computing certain hom spaces of the category of $\displaystyle \bigoplus_{m\geq 0}(\A_m, \A_n)-$bimodules. A complete description of the category is an open problem.
\end{abstract}

\maketitle
\tableofcontents
\bibliographystyle{amsalpha}

\section{Introduction}
    Representation theory of symmetric groups is a classical and rich subject. The study of their irreducible representations has been initiated by Schur (cf. \cite{MR1820589}). Today we have a well-developed theory of representations of symmetric groups in characteristic zero. Symmetric groups form a tower of groups 
    \[ 
    S_0\longhookrightarrow S_1 \longhookrightarrow S_2 \longhookrightarrow \ldots \longhookrightarrow S_n \longhookrightarrow \dots
    \]  
    via the inclusions $S_n\longhookrightarrow S_{n+1}$, $s_i=(i, i+1) \mapsto s_i$, for example. Hence once the representation category of symmetric groups is understood, one can study induction and restriction functors on these categories.
   
   The functors $\Ind$ and $\Res$ are biadjoint, and are related through Mackey's theorem. In 2010, using this relation Khovanov described a category (known as Heisenberg category) which governs the natural transformations between induction and restriction functors between representations of the symmetric groups (cf. \cite{MR3205569}).
   
   Today, we have more general Heisenberg categories, which are quantized, have a central charge $c$, and depend on a choice of a graded Frobenius superalgebra $F$, where $\{F\wr S_n\}_{n\in \mathbb{N}}$ play the role of the symmetric groups in Khovanov's construction, see \textit{e.g.}, 
   \cite{MR3648512,Savage-Frob-Heisenberg, Brundan_2018, brundan2019degenerate, Brundan_2020, brundan2021foundations} . 
    
    In this paper, we investigate induction and restriction functors arising from a different tower, namely the tower of iterated wreath products of symmetric groups of order two. Intuitively, we work with the tower $\A_0 \longhookrightarrow \A_1 \longhookrightarrow \A_2 \longhookrightarrow \ldots$, where $\A_0=\{\id\}$, $\A_1=S_2$ and $\A_n=S_2\wr \A_{n-1}$ for $n\geq 2$ and study the categories $\mathcal{A}_n$ with objects $\displaystyle \bigoplus_{m\geq 0}(\A_m, \A_n)-$bimodules and with morphisms bimodule homomorphisms. Induction and restriction functors can be identified with tensoring with certain bimodules, therefore the morphisms in the category $\mathcal{A}_n$ correspond to natural transformations between these functors. We will denote by $\mathcal{C}_n$ the $\mathbb{C}$-linear additive category whose objects are generated by compositions of induction and restriction functors between groups $\A_k$, which start from $\A_n$. Morphisms of $\mathcal{C}_n$ are natural transformations between the induction and restriction functors.

    Since every finite-dimensional group algebra is a Frobenius algebra, the tower $\{\A_n\}_{n\in \mathbb{N}}$ consists of Frobenius algebras. However, this tower is fundamentally different than the towers in Savage's papers in that we do not take wreath products with symmetric groups increasing in size. Instead, we are iterating the wreath product.
    
    We mention some other work related to the iterated wreath products for the interested reader. In \cite{OOR}, the authors consider the iterated wreath products of $\mathbb{Z}/n\mathbb{Z}$ and describe a correspondence between their irreducible representations and certain rooted trees. Such a correspondence allows them to determine the Bratelli diagram for their tower of groups. In \cite{IW-cyclic}, first author and Angela Wu generalize these results to iterated wreath products of cyclic groups of various orders. As a follow up, in \cite{IW-symm} they generalize their results to iterated wreath products of symmetric groups of various sizes. Our focus is not the representations themselves but induction and restriction between them, as well as natural transformations between induction and restriction.
    In \cite{Im-Khovanov-Rozansky}, the authors introduce a notion of foams and relations on them to interpret functors and natural transformations in categories of representations of certain iterated wreath products, leading towards a discussion of patched surfaces with certain defects and their connections to separable field extensions and Galois theory. 
    
    In order to describe the category $\mathcal{C}_n$ or $\mathcal{A}_n$, one needs to establish certain structural properties of the groups $\A_n$. 
    In Section~\ref{section:iterated-wreath-products}, we provide background on rooted trees and its connection to the construction of iterated wreath products.  
    In Section~\ref{section:structural-properties}, we prove results about the right cosets, centers, centralizer algebras, and double cosets of $\A_n$'s. In Section~\ref{section:natural-transformation-ind-res}, 
    we consider $\Hom_{\mathcal{A}_n}({\Ind^{k_1}\Res^{l_1}, \Ind^{k_2}\Res^{l_2}})$, where $k_1-l_1=k_2-l_2$, $l_1\leq n$ and $l_2\leq n$. We provide an explicit vector space basis for $\End_{\mathcal{A}_n}(\Ind^{k}\Res^{l})$ in Proposition~\ref{prop:vs-basis-general}, and we construct a generating set of $\End_n(\Ind^k)$ in Theorem~\ref{theorem:alg-gens}. 
    In Section~\ref{section:future-direction}, we conclude by discussing many interesting open problems, in particular, in connection to the Heisenberg categories.


\section{Rooted trees and iterated wreath products}
\label{section:iterated-wreath-products}
    Iterated wreath products of cyclic groups or permutation groups can be seen as automorphism groups of certain rooted trees (cf. \cite{IW-cyclic,IW-symm,OOR,Im-Khovanov-Rozansky}).
    A {\color{purple}\it rooted tree} is a connected simple graph with no cycles and with a distinguished vertex, which is called a {\color{purple}\it root}. A vertex is in the $j$-th level of a rooted tree if it is at distance $j$ from the root. The {\color{purple}\it branching factor of a vertex} is its number of children, and a {\color{purple}\it leaf} is a vertex with branching factor zero. 
    
    A {\color{purple}\it complete binary tree $T_n$ 
    of height $n$} is a rooted tree with all of its leaves at level $n$, and whose vertices that are not leaves all have branching factor $2$.
    
    A {\color{purple}\it labeled complete binary tree} is a complete binary tree whose children are labeled $1,2,\ldots, 2^n$. Without loss of generality, we will also denote a labeled tree by $T_n$.

    
    The wreath product $\A_n := S_2\wr\cdots \wr S_2$ of the group $\A_{n-1}$ with $S_2$ is $(\A_{n-1}\times \A_{n-1})\times S_2$ with the multiplication given by  $(a;\sigma)(b;\lambda)= (ab^{\sigma};\sigma\lambda)$, where $b^{\sigma}$ corresponds to swapping the components of an element $b\in \A_{n-1}\times \A_{n-1}$ if $\sigma$ is not the identity of $S_2$, and $b^\sigma=b$ if $\sigma$ is the identity of $S_2$. The symmetric group of degree $2^n$ has as its Sylow $2$-subgroup the $n$-step iterated wreath product of $S_n$.
    Writing $e'$ to be the identity element in the wreath product $\A_{n-1}$ and $e$ to be the identity element in $S_2$, the group $S_2\cong \{(e';\sigma): \sigma\in S_2\}$ and $\A_{n-1}\cong \{(a;e): a\in \A_{n-1} \}$. This identification shows that the wreath product is precisely the semi-direct product of $\A_{n-1}$ by $S_2$. 
    
    Another way to think of the wreath product is through automorphism groups of rooted trees. An automorphism of a tree $T_n$ with vertex set $V$ is a bijection $f:V\lra V$ such that $u,v\in V$ are {\color{purple}\it adjacent} if and only if $f(u)$ and $f(v)$ are adjacent. These automorphisms of $T_n$ form a group under composition. So the wreath product $S_2\wr\cdots \wr S_2$ of the group $\A_{n-1}$ with $S_2$ may be identified with the automorphism group of complete binary tree of height $n$ that permutes the $2$ children of each node. 
 
    That is, let $\A_1=S_2$ and $\A_n= \A_{n-1}\wr S_2$ for $n\geq 2$. The iterated wreath product $\A_n$ may be seen as the automorphism group of $T_n$, since $\A_{n-1}\wr S_2$ means take two copies of $\A_{n-1}$ and permute them. One can see the two branches coming out of the root of $T_n$ as two copies of the tree $T_{n-1}$, and an automorphism of $T_n$ can leave those two branches fixed, or swap them (which then becomes the action of $S_2$ on two copies of the wreath product $\A_{n-1}$).
    
    We will denote the root of a rooted tree with an empty circle. See, \textit{e.g.},
    \[ 
    \A_2 \simeq \Aut(T_2) = \Aut\Bigg(\quad 
    \raisebox{10pt}{\xymatrix@-1.65pc{
    *={\color{blue}{\bullet}\save+<0ex,2ex>\restore}
    & &
    *={\color{blue}{\bullet}\save+<0ex,2ex>\restore}
    & &  
    *={\color{blue}{\bullet}\save+<0ex,2ex>\restore}
    & & 
    *={\color{blue}{\bullet}\save+<0ex,2ex>\restore} 
    \frm{}
    \\
    & \ar@{-}@[blue][lu]
    *={} \save+<0ex,0ex>*={\bullet} \restore
    \ar@{-}@[blue][ru]
    & & & & 
    \ar@{-}@[blue][lu]
    *={\color{blue}{\bullet}}  \save+<0ex,0ex>*={\color{blue}{\bullet} \restore}
    \ar@{-}@[blue][ru] 
    & 
    \frm{}\\
    & & &\ar@{-}@[blue][llu] 
    *={\color{blue}{\bullet}} \save+<0ex,0ex>*={\color{blue}\circ} \restore
    \ar@{-}@[blue][rru]& & & 
    \frm{} \\ 
    }} \quad \Bigg).
    \]
    
    With the tree picture in mind, there are two natural choices for embedding the wreath product $\A_n$ into $\A_{n+1}$. One can either add new branches to the top of a tree, or one can take two copies of a tree of height $n$ and connect them through a new root. 
    
    The first one preserves the root of the tree, hence we call it \emph{\color{purple}tree embedding} of the wreath product and denote it as $i_T:\A_n  \longhookrightarrow \A_{n+1}$.
    For example, 
    \[ 
    \textcolor{blue}{\A_2} \stackrel{i_T}{\longhookrightarrow} \A_3 \simeq \Aut(T_3) = \Aut\Bigg( \hspace{0.15cm}
    \raisebox{20pt}{\xymatrix@-1.25pc{
    *={{\bullet} \save+<0ex,2ex> \restore}
    & & 
    *={{\bullet} \save+<0ex,2ex> \restore}
    & & 
    *={{\bullet} \save+<0ex,2ex> \restore}
    & &
    *={{\bullet} \save+<0ex,2ex> \restore}
    & &
    *={\bullet} \save+<0ex,2ex> \restore
    & &
    *={\bullet} \save+<0ex,2ex> \restore
    & &
    *={\bullet} \save+<0ex,2ex> \restore
    & &
    *={\bullet} \save+<0ex,2ex> \restore \frm{}
     \\
    &   \ar@{-}[ul] 
    *={{\bullet}} \save+<0ex,0ex>*={\color{blue}\bullet} \restore
    \ar@{-}[ur]
    & & & & \ar@{-}[ul] 
    *={\bullet} \save+<0ex,0ex>*={\color{blue}\bullet} \restore
    \ar@{-}[ur]  & & & & \ar@{-}[ul] 
    *={\bullet} \save+<0ex,0ex>*={\color{blue}\bullet \restore}
    \ar@{-}[ur] & & & & \ar@{-}[ul] 
    *={\bullet} \save+<0ex,0ex>*={\color{blue}\bullet \restore}
    \ar@{-}[ur] & 
    \frm{}  \\
    & & & \ar@{-}@[blue][ull] 
    *={\color{blue}{\bullet}} \save+<0ex,0ex>*={\color{blue}\bullet} \restore
    \ar@{-}@[blue][urr]
    & & & & & & & & 
    \ar@{-}@[blue][ull]
    *={\color{blue}\bullet} \save+<0ex,0ex>*={\color{blue}\bullet \restore}
    \ar@{-}@[blue][urr] & & & \frm{} \\
    & & & & & & & \ar@{-}@[blue][ullll] 
    *={\color{blue}{\bullet}} \save+<0ex,0ex>*={\color{blue}\circ} \restore
    \ar@{-}@[blue][urrrr]& & & & & & &  \\ 
    }} \hspace{0.15cm} \Bigg)
    \] 
    
    The second one, although it does not preserve the root of $T_n$, is more convenient to work with when one wants to identify elements of $\A_n$ with permutations of the leaves. If we label the leaves of a tree from $1$ to $2^n$, then each tree automorphism can be seen as a permutation in $S_{2^n}$. In the second embedding the expression of a permutation and its image are the same in cycle notation, since the labeling of the leaves is preserved. Hence this embedding will be called \emph{\color{purple}permutation embedding} and will be denoted as $i_p:\A_n \longhookrightarrow \A_{n+1}$. For example, 
    \[ 
    \textcolor{blue}{\A_2} \stackrel{i_p}{\longhookrightarrow} \A_3 \simeq \Aut(T_3) = \Aut\Bigg( \quad
    \raisebox{20pt}{\xymatrix@-1.5pc{
    *={\color{blue}{\bullet} \save+<0ex,2ex> \restore}
    & & 
    *={\color{blue}{\bullet} \save+<0ex,2ex> \restore}
    & & 
    *={\color{blue}{\bullet} \save+<0ex,2ex> \restore}
    & &
    *={\color{blue}{\bullet} \save+<0ex,2ex> \restore}
    & &
    *={\bullet} \save+<0ex,2ex> \restore
    & &
    *={\bullet} \save+<0ex,2ex> \restore
    & &
    *={\bullet} \save+<0ex,2ex> \restore
    & &
    *={\bullet} \save+<0ex,2ex> \restore \frm{}
     \\
    &   \ar@{-}@[blue][ul] 
    *={\color{blue}{\bullet}} \save+<0ex,0ex>*={\color{blue}{\bullet} \restore}
    \ar@{-}@[blue][ur]
    & & & & \ar@{-}@[blue][ul] 
    *={\color{blue}{\bullet}} \save+<0ex,0ex>*={\color{blue}{\bullet} \restore}
    \ar@{-}@[blue][ur]  & & & & \ar@{-}[ul] 
    *={\color{black}{\bullet}} \save+<0ex,0ex>*={\bullet} \restore
    \ar@{-}[ur] & & & & \ar@{-}[ul] 
    *={\bullet} \save+<0ex,0ex>*={\bullet} \restore
    \ar@{-}[ur] & 
    \frm{}  \\
    & & & \ar@{-}@[blue][ull] 
    *={\color{blue}{\bullet}} \save+<0ex,0ex>*={\bullet} \restore
    \ar@{-}@[blue][urr]
    & & & & & & & & 
    \ar@{-}[ull]
    *={\bullet} \save+<0ex,0ex>*={\bullet} \restore
    \ar@{-}[urr] & & & \frm{} \\
    & & & & & & & \ar@{-}[ullll] 
    *={\color{blue}{\bullet}} \save+<0ex,0ex>*={\circ} \restore
    \ar@{-}[urrrr]& & & & & & &  \\ 
    }} \quad \Bigg).
    \]

    Automorphisms of $T_n$ are determined by whether two children of each vertex are swapped or not. We will represent these swaps by dotted arrows on the diagram of a rooted tree. When reading a tree automorphism from the tree diagram, we apply the swaps from the bottom-most layer to the top-most layer, where the root of the tree is considered to be at the bottom. 
    
    From this point forward, we will work with the permutation embedding, and write $i:= i_p$. 
  
    \begin{example}\label{ex:A1-autom-tree}  
    The wreath product $\A_1 = S_2 = \{ e, (1,2)\}$ 
    can be thought of as the set of automorphisms of the tree 
\begin{center}
{\Large $T_1$ $:$} $\quad$ 
$\xymatrix@-1.25pc{
   *={\bullet}\save+<0ex,2ex>*={1}\restore  &  &*={\bullet}\save+<0ex,2ex>*={2}\restore&  
   \frm{}\\
   &\save+<0ex,0ex>*={\circ} \restore \ar@{-}[lu] \ar@{-}[ru]& & 
 }$
\end{center}
   where the vertices of the top layer are labeled $1$ and $2$ when reading from left to right. 
    That is, there is an isomorphism $\A_1\lra \Aut(T_1)$ of groups, 
    where 
\begin{align*}
   e &\mapsto  
\xymatrix@-1.25pc{
   *={\bullet}\save+<0ex,2ex>*={1}\restore  & &*={\bullet}\save+<0ex,2ex>*={2}\restore &  
   \frm{}\\
   &\save+<0ex,0ex>*={\circ} \restore \ar@{-}[lu] \ar@{-}[ru]& & \frm{} \\ 
 }:=\quad 
     \left( 
    \xymatrix@-1.25pc{
   *={\bullet}\save+<0ex,2ex>*={1}\restore  &  
   &*={\bullet}  \save+<0ex,2ex>*={2}\restore&  
   \frm{}\\
   &\save+<0ex,0ex>*={\circ} \restore \ar@{-}[lu] \ar@{-}[ru]& & \frm{} \\
    }
    \mapsto  \quad 
\xymatrix@-1.25pc{
   *={\bullet}\save+<0ex,2ex>*={1}\restore  &  &*={\bullet}\save+<0ex,2ex>*={2}\restore&  
   \frm{}\\
   &\save+<0ex,0ex>*={\circ} \restore  \ar@{-}[lu] \ar@{-}[ru]& & \frm{} \\
 }\right),  \\ 
    (1,2) &\mapsto 
        \xymatrix@-1.25pc{
   *={\bullet}\save+<0ex,2ex>*={1}\restore \ar@{<..>}@/^0.5pc/[rr] &  
   &*={\bullet}  \save+<0ex,2ex>*={2}\restore&  
   \frm{}\\
   &\save+<0ex,0ex>*={\circ} \restore \ar@{-}[lu] \ar@{-}[ru]& & \frm{} \\
    } := 
    \quad 
    \left( 
    \xymatrix@-1.25pc{
   *={\bullet}\save+<0ex,2ex>*={1}\restore  &  
   &*={\bullet}  \save+<0ex,2ex>*={2}\restore&  
   \frm{}\\
   &\save+<0ex,0ex>*={\circ} \restore \ar@{-}[lu] \ar@{-}[ru]& & \frm{} \\
    }
    \mapsto  \quad 
\xymatrix@-1.25pc{
   *={\bullet}\save+<0ex,2ex>*={2}\restore  &  &*={\bullet}\save+<0ex,2ex>*={1}\restore&  
   \frm{}\\
   &\save+<0ex,0ex>*={\circ} \restore  \ar@{-}[lu] \ar@{-}[ru]& & \frm{} \\
 }\right). \\ 
\end{align*}
\end{example}

\begin{example}\label{ex:A2-autom-tree} 
    Consider $\A_2=S_2\wr S_2=(S_2\times S_2)\rtimes S_2=\{(a,b;c)|(a,b)\in S_2\times S_2, c\in S_2\}$. Given the complete binary tree 
\begin{center}

{\Large $T_2:$ $\qquad$} 
    $\xymatrix@-1.25pc{
    *={\bullet}\save+<0ex,2ex>*={1}\restore
    & &
    *={\bullet}\save+<0ex,2ex>*={2}\restore
    & &  
    *={\bullet}\save+<0ex,2ex>*={3}\restore
    & & 
    *={\bullet}\save+<0ex,2ex>*={4}\restore 
    \frm{}
    \\
    &  \ar@{-}[lu]
    *={\bullet} \save+<0ex,0ex>*={\bullet} \restore
    \ar@{-}[ru] 
    & & & & 
    \ar@{-}[lu]
    *={\bullet} \save+<0ex,0ex>*={\bullet} \restore
    \ar@{-}[ru] 
    & 
    \frm{}\\
    & & &\ar@{-}[llu] 
    \save+<0ex,0ex>*={\circ} \restore  \save+<0ex,0ex>*={\circ} \restore
    \ar@{-}[rru]& & & 
    \frm{} \\ 
    }$, 
\end{center}
    the map $\A_2\stackrel{\simeq}{\lra}\Aut(T_2)$ is a group homomorphism, 
    where the isomorphism is indicated via sending 
    \begin{align*}
    ((1,2),e;e)\mapsto 
        &\xymatrix@-1.5pc{
    *={\bullet} \save+<0ex,2ex>*={1} \restore 
    \ar@{<..>}@/^0.25pc/[rr]& &
    *={\bullet} \save+<0ex,2ex>*={2} \restore 
    & &  
    *={\bullet} \save+<0ex,2ex>*={3} \restore
    & & 
    *={\bullet} \save+<0ex,2ex>*={4} \restore 
    \frm{} \\
    &  \ar@{-}[lu]
    *={\bullet} \save+<0ex,0ex>*={\bullet} \restore
    \ar@{-}[ru] 
    & & & & 
    \ar@{-}[lu]
    *={\bullet} \save+<0ex,0ex>*={\bullet} \restore
    \ar@{-}[ru] 
    & \frm{} \\
    & & &\ar@{-}[llu] 
    *={\circ} \save+<0ex,0ex>*={\circ} \restore
    \ar@{-}[rru]& & & \frm{} \\ 
    } 
    \quad \mbox{ and }
    \quad 
    (e,e;(1,2))\mapsto 
    \xymatrix@-1.5pc{
    *={\bullet} \save+<0ex,2ex>*={1} \restore
    & &
    *={\bullet} \save+<0ex,2ex>*={2} \restore
    & &  
    *={\bullet} \save+<0ex,2ex>*={3} \restore
    & & 
    *={\bullet} \save+<0ex,2ex>*={4} \restore 
    \frm{} \\
    &  \ar@{-}[lu]
    *={\bullet} \save+<0ex,0ex>*={\bullet} \restore
    \ar@{-}[ru] \ar@{<..>}@/^.25pc/[rrrr]
    & & & & 
    \ar@{-}[lu]
    *={\bullet} \save+<0ex,0ex>*={\bullet} \restore
    \ar@{-}[ru] 
    & 
    \frm{} \\
    & & &\ar@{-}[llu] 
    \save+<0ex,0ex>*={\circ} \restore  \save+<0ex,0ex>*={\circ} \restore
    \ar@{-}[rru]& & & \frm{} \\ 
    }. 
    \end{align*}

    \input{fig2_1}    
    
    We thus have an embedding of $i:\A_1\longhookrightarrow \A_2$, where $i(\A_1)$ maps to the automorphisms of $T_2$ that fixes vertices $3$ and $4$, \textit{i.e.}, the right branch in the blue circle is fixed. See Figure~\ref{fig2_1}. 
    This gives us a map 
    $\A_1\longhookrightarrow \Aut(T_2)$, which is given by 
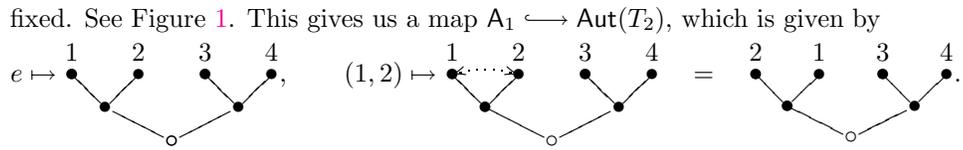
    \[
    e \mapsto 
    \xymatrix@-1.45pc{
    *={\bullet} \save+<0ex,2ex>*={1} \restore
    & &
    *={\bullet} \save+<0ex,2ex>*={2} \restore
    & &  
    *={\bullet} \save+<0ex,2ex>*={3} \restore
    & & 
    *={\bullet} \save+<0ex,2ex>*={4} \restore \frm{} \\
    &  \ar@{-}[lu]
    *={\bullet} \save+<0ex,0ex>*={\bullet} \restore
    \ar@{-}[ru] 
    & & & & 
    \ar@{-}[lu]
    *={\bullet} \save+<0ex,0ex>*={\bullet} \restore
    \ar@{-}[ru] 
    & \frm{} \\
    & & &\ar@{-}[llu] 
    \save+<0ex,0ex>*={\circ} \restore \save+<0ex,0ex>*={\circ} \restore
    \ar@{-}[rru]& & & \frm{} \\ 
    }, 
    \qquad 
    (1,2) \mapsto 
    \xymatrix@-1.45pc{
    *={\bullet} \save+<0ex,2ex>*={1} \restore \ar@{<..>}@/^.15pc/[rr]
    & &
    *={\bullet} \save+<0ex,2ex>*={2} \restore
    & &  
    *={\bullet} \save+<0ex,2ex>*={3} \restore
    & & 
    *={\bullet} \save+<0ex,2ex>*={4} \restore
    \frm{} \\
    &  \ar@{-}[lu]
    *={\bullet} \save+<0ex,0ex>*={\bullet} \restore
    \ar@{-}[ru] 
    & & & & 
    \ar@{-}[lu]
    *={\bullet} \save+<0ex,0ex>*={\bullet} \restore
    \ar@{-}[ru] 
    & \\
    & & &\ar@{-}[llu] 
    *={\circ} \save+<0ex,0ex>*={\circ} \restore
    \ar@{-}[rru]& & & \frm{} \\ 
    } 
    \quad 
    = 
    \quad 
    \xymatrix@-1.5pc{
    *={\bullet} \save+<0ex,2ex>*={2} \restore
    & &
    *={\bullet} \save+<0ex,2ex>*={1} \restore
    & &  
    *={\bullet} \save+<0ex,2ex>*={3} \restore
    & & 
    *={\bullet} \save+<0ex,2ex>*={4} \restore
    \frm{}\\
    &  \ar@{-}[lu]
    *={\bullet} \save+<0ex,0ex>*={\bullet} \restore
    \ar@{-}[ru]
    & & & & 
    \ar@{-}[lu]*={\bullet} \save+<0ex,0ex>*={\bullet} \restore
    \ar@{-}[ru] 
    & \frm{} \\
    & & &\ar@{-}[llu] 
    *={\circ} \save+<0ex,0ex>*={\circ} \restore
    \ar@{-}[rru]& & & \frm{} \\ 
    }. 
    \] 
    Note that there is also an embedding $j:\A_2\longhookrightarrow S_4$ of the wreath product $\A_2$ into the symmetric group $S_4$ on $4$ letters, 
    where one just keeps track of the labels on the top of the labeled tree $T_2$. For example
    $((1,2),e;e)\mapsto (1,2)$ and 
    $(e,e;(1,2))\mapsto (1,3)(2,4)$.
    \end{example}
    

    \begin{example}\label{ex:A1-in-A2}
    The group $\A_3$ has a canonical identification with the automorphism group of the tree 
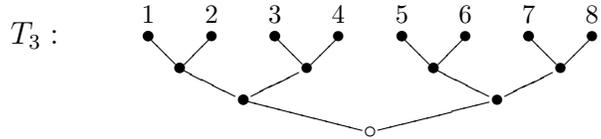
\begin{center}
    {\Large $T_3:$ $\qquad$} 
    $\xymatrix@-1.5pc{
    *={\bullet} \save+<0ex,2ex>*={1} \restore
    & & 
    *={\bullet} \save+<0ex,2ex>*={2} \restore
    & & 
    *={\bullet} \save+<0ex,2ex>*={3} \restore
    & &
    *={\bullet} \save+<0ex,2ex>*={4} \restore
    & &
    *={\bullet} \save+<0ex,2ex>*={5} \restore
    & &
    *={\bullet} \save+<0ex,2ex>*={6} \restore
    & &
    *={\bullet} \save+<0ex,2ex>*={7} \restore
    & &
    *={\bullet} \save+<0ex,2ex>*={8} \restore \frm{} \\
    &   \ar@{-}[ul] 
    *={\bullet} \save+<0ex,0ex>*={\bullet} \restore
    \ar@{-}[ur]
    & & & & \ar@{-}[ul] 
    *={\bullet} \save+<0ex,0ex>*={\bullet} \restore
    \ar@{-}[ur]  & & & & \ar@{-}[ul] 
    *={\bullet} \save+<0ex,0ex>*={\bullet} \restore
    \ar@{-}[ur] & & & & \ar@{-}[ul] 
    *={\bullet} \save+<0ex,0ex>*={\bullet} \restore
    \ar@{-}[ur] & 
    \frm{}  \\
    & & & \ar@{-}[ull] 
    *={\bullet} \save+<0ex,0ex>*={\bullet} \restore
    \ar@{-}[urr]
    & & & & & & & & 
    \ar@{-}[ull]
    *={\bullet} \save+<0ex,0ex>*={\bullet} \restore
    \ar@{-}[urr] & & & \frm{} \\
    & & & & & & & \ar@{-}[ullll] 
    *={\circ} \save+<0ex,0ex>*={\circ} \restore
    \ar@{-}[urrrr]& & & & & & & \frm{} \\ 
    }$.
\end{center}  
    Furthermore, the permutation embedding $i: \A_2  \longhookrightarrow \A_3$ gives rise to the identification between $\A_2$ and the automorphisms of $T_3$ that fixes the children of level $2$ vertices fixed on the right branch: 
    \begin{align*}
    ((1,2),e;e)
    \hspace{2mm}
    \mbox{\Large $\mapsto$} 
    \hspace{2mm}
    &
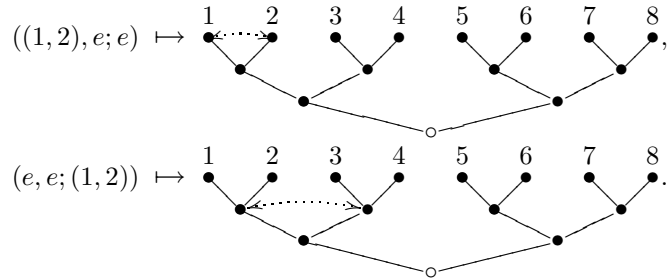
\xymatrix@-1.5pc{
    *={\bullet} \save+<0ex,2ex>*={1} \restore 
    \ar@{<..>}@/^0.15pc/[rr]
    & & 
    *={\bullet} \save+<0ex,2ex>*={2} \restore
    & & 
    *={\bullet} \save+<0ex,2ex>*={3} \restore
    & &
    *={\bullet} \save+<0ex,2ex>*={4} \restore
    & &
    *={\bullet} \save+<0ex,2ex>*={5} \restore
    & &
    *={\bullet} \save+<0ex,2ex>*={6} \restore
    & &
    *={\bullet} \save+<0ex,2ex>*={7} \restore
    & &
    *={\bullet} \save+<0ex,2ex>*={8} \restore \frm{} \\
    &   \ar@{-}[ul] 
    *={\bullet} \save+<0ex,0ex>*={\bullet} \restore
    \ar@{-}[ur]
    & & & & \ar@{-}[ul] 
    *={\bullet} \save+<0ex,0ex>*={\bullet} \restore
    \ar@{-}[ur]  & & & & \ar@{-}[ul] 
    *={\bullet} \save+<0ex,0ex>*={\bullet} \restore
    \ar@{-}[ur] & & & & \ar@{-}[ul] 
    *={\bullet} \save+<0ex,0ex>*={\bullet} \restore
    \ar@{-}[ur] & 
    \frm{}  \\
    & & & \ar@{-}[ull] 
    *={\bullet} \save+<0ex,0ex>*={\bullet} \restore
    \ar@{-}[urr]
    & & & & & & & & 
    \ar@{-}[ull]
    *={\bullet} \save+<0ex,0ex>*={\bullet} \restore
    \ar@{-}[urr] & & & \frm{} \\
    & & & & & & & \ar@{-}[ullll] 
    *={\circ} \save+<0ex,0ex>*={\circ} \restore
    \ar@{-}[urrrr]& & & & & & & \frm{} \\ 
    },  \\ 
    (e,e;(1,2))
    \hspace{2mm}
    \mbox{\Large $\mapsto$} 
    \hspace{2mm}
    &\xymatrix@-1.5pc{
    *={\bullet} \save+<0ex,2ex>*={1} \restore
    & & 
    *={\bullet} \save+<0ex,2ex>*={2} \restore
    & & 
    *={\bullet} \save+<0ex,2ex>*={3} \restore
    & &
    *={\bullet} \save+<0ex,2ex>*={4} \restore
    & &
    *={\bullet} \save+<0ex,2ex>*={5} \restore
    & &
    *={\bullet} \save+<0ex,2ex>*={6} \restore
    & &
    *={\bullet} \save+<0ex,2ex>*={7} \restore
    & &
    *={\bullet} \save+<0ex,2ex>*={8} \restore \frm{} \\
    &   \ar@{-}[ul] 
    *={\bullet} \save+<0ex,0ex>*={\bullet} \restore
    \ar@{-}[ur]
    \ar@{<..>}@/^0.25pc/[rrrr]
    & & & & \ar@{-}[ul] 
    *={\bullet} \save+<0ex,0ex>*={\bullet} \restore
    \ar@{-}[ur]  & & & & \ar@{-}[ul] 
    *={\bullet} \save+<0ex,0ex>*={\bullet} \restore
    \ar@{-}[ur] & & & & \ar@{-}[ul] 
    *={\bullet} \save+<0ex,0ex>*={\bullet} \restore
    \ar@{-}[ur] & 
      \\
    & & & \ar@{-}[ull] 
    *={\bullet} \save+<0ex,0ex>*={\bullet} \restore
    \ar@{-}[urr]
    & & & & & & & & 
    \ar@{-}[ull]
    *={\bullet} \save+<0ex,0ex>*={\bullet} \restore
    \ar@{-}[urr] & & & \\
    & & & & & & & \ar@{-}[ullll] 
    *={\circ} \save+<0ex,0ex>*={\circ} \restore
    \ar@{-}[urrrr]& & & & & & & \frm{} \\ 
    }. 
    \end{align*} 
    
    \end{example}
    Therefore, $\A_n$ is isomorphic to the automorphism group of a complete rooted binary tree $T_n$ of height $n$, and $\A_{n-1}$ is isomorphic to the subgroup of the automorphism group of $T_n$, which fixes the children of level $n-1$ vertices. 
    
    We will denote by $\widehat{\A}_n$ the image of $\A_n$ inside $\A_{n+1}$ induced by the map \[
    \left\{1,2,\dots,2^n\right\} \lra  \left\{2^n+1,2^n+2,\dots,2^{n+1}\right\},
    \] 
    which sends the label $k$ to the label $2^n+k$, where $1\leq k\leq 2^n$, and by ${\beta}_{n+1}$ the permutation $(1,2^n+1)(2,2^n+2)\cdots(2^n,2^{n+1})$.

Note that in terms of the tree $T_{n+1}$, $\A_n$ is the automorphism of the children of the leftmost branch growing from level $0$, $\widehat{\A}_n$ is the automorphism of the children of the rightmost branch growing from level $0$, and ${\beta}_{n+1}$ is the automorphism of the tree $T_{n+1}$ which swaps $\A_n$ and $\widehat{\A}_n$.

\begin{example}
When $n=2$,  $\widehat{\A}_2$ is the automorphism of the subtree of $T_3$ indicated by the dotted edges: 
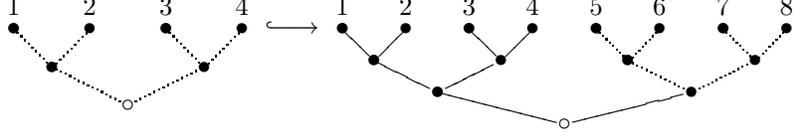
\[ 
\xymatrix@-1.3pc{
    *={\bullet} \save+<0ex,2ex>*={1} \restore
    & &
    *={\bullet} \save+<0ex,2ex>*={2} \restore
    & &  
    *={\bullet} \save+<0ex,2ex>*={3} \restore
    & & 
    *={\bullet} \save+<0ex,2ex>*={4} \restore
    \frm{} \\
    &  \ar@{..}[lu]
    *={\bullet} \save+<0ex,0ex>*={\bullet} \restore
    \ar@{..}[ru] 
    & & & & 
    \ar@{..}[lu]
    *={\bullet} \save+<0ex,0ex>*={\bullet} \restore
    \ar@{..}[ru] 
    & \frm{} \frm{} \\
    & & &\ar@{..}[llu] 
    *={\circ} \save+<0ex,0ex>*={\circ} \restore
    \ar@{..}[rru]& & & \frm{} \\ 
    }
    \hspace{2mm}
    \mbox{\Large $\longhookrightarrow$} 
    \hspace{2mm}
    \xymatrix@-1.5pc{
    *={\bullet} \save+<0ex,2ex>*={1} \restore
    & & 
    *={\bullet} \save+<0ex,2ex>*={2} \restore
    & & 
    *={\bullet} \save+<0ex,2ex>*={3} \restore
    & &
    *={\bullet} \save+<0ex,2ex>*={4} \restore
    & &
    *={\bullet} \save+<0ex,2ex>*={5} \restore
    & &
    *={\bullet} \save+<0ex,2ex>*={6} \restore
    & &
    *={\bullet} \save+<0ex,2ex>*={7} \restore
    & &
    *={\bullet} \save+<0ex,2ex>*={8} \restore \frm{} \\
    &   \ar@{-}[ul] 
    *={\bullet} \save+<0ex,0ex>*={\bullet} \restore
    \ar@{-}[ur]
    & & & & \ar@{-}[ul] 
    *={\bullet} \save+<0ex,0ex>*={\bullet} \restore
    \ar@{-}[ur]  & & & & \ar@{..}[ul] 
    *={\bullet} \save+<0ex,0ex>*={\bullet} \restore
    \ar@{..}[ur] & & & & \ar@{..}[ul] 
    *={\bullet} \save+<0ex,0ex>*={\bullet} \restore
    \ar@{..}[ur] & 
    \frm{}  \\
    & & & \ar@{-}[ull] 
    *={\bullet} \save+<0ex,0ex>*={\bullet} \restore
    \ar@{-}[urr]
    & & & & & & & & 
    \ar@{..}[ull]
    *={\bullet} \save+<0ex,0ex>*={\bullet} \restore
    \ar@{..}[urr] & & & \frm{} \\
    & & & & & & & \ar@{-}[ullll] 
    *={\circ} \save+<0ex,0ex>*={\circ} \restore
    \ar@{-}[urrrr]& & & & & & & \frm{} \\ 
    },  \\ 
\] 
and ${\beta}_3$ is the tree automorphism given below: 
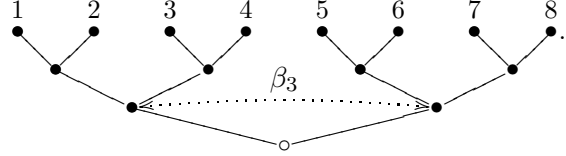
\[ 
    \xymatrix@-1.3pc{
    *={\bullet} \save+<0ex,2ex>*={1} \restore
    & & 
    *={\bullet} \save+<0ex,2ex>*={2} \restore
    & & 
    *={\bullet} \save+<0ex,2ex>*={3} \restore
    & &
    *={\bullet} \save+<0ex,2ex>*={4} \restore
    & &
    *={\bullet} \save+<0ex,2ex>*={5} \restore
    & &
    *={\bullet} \save+<0ex,2ex>*={6} \restore
    & &
    *={\bullet} \save+<0ex,2ex>*={7} \restore
    & &
    *={\bullet} \save+<0ex,2ex>*={8} \restore \frm{} \\
    &   \ar@{-}[ul] 
    *={\bullet} \save+<0ex,0ex>*={\bullet} \restore
    \ar@{-}[ur]
    & & & & \ar@{-}[ul] 
    *={\bullet} \save+<0ex,0ex>*={\bullet} \restore
    \ar@{-}[ur]  & & & & \ar@{-}[ul] 
    *={\bullet} \save+<0ex,0ex>*={\bullet} \restore
    \ar@{-}[ur] & & & & \ar@{-}[ul] 
    *={\bullet} \save+<0ex,0ex>*={\bullet} \restore
    \ar@{-}[ur] & 
    \frm{}  \\
    & & & \ar@{-}[ull] 
    *={\bullet} \save+<0ex,0ex>*={\bullet} \restore
    \ar@{-}[urr] 
    \ar@{<..>}@/^0.25pc/[rrrrrrrr]^{\mbox{\large ${\beta}_3$}}
    & & & & & & & & 
    \ar@{-}[ull]
    *={\bullet} \save+<0ex,0ex>*={\bullet} \restore
    \ar@{-}[urr] 
    & & & \frm{} \\
    & & & & & & & \ar@{-}[ullll] 
    *={\circ} \save+<0ex,0ex>*={\circ} \restore 
    \ar@{-}[urrrr]& & & & & & & \frm{} \\ 
    }.  \\ 
\] 
\end{example}
    
    We will make a systematic use of the following observations in the rest of the paper:

\begin{lemma}\label{enum:labelsets-L1L2} 
The groups $\A_n$ and $\widehat{\A}_n$ commute.
\end{lemma} 

\begin{proof}
$\A_n$ only acts on the label set $L_1=\{1,2,\dots,2^n\}$ and $\widehat{\A}_n$ only acts on the label set $L_2=\{2^n+1,2^n+2,\ldots,2^{n+1}\}$. Hence, acting on disjoint labels, these two subgroups commute with each other, \textit{i.e.},  $\A_n\widehat{\A}_n = \widehat{\A}_n \A_n$. 
\end{proof}

\begin{remark}\label{enum:beta-link} 
The automorphism ${\beta}_{n+1}$ is the only link between the two branches $\A_n$ and $\widehat{\A}_n$.
\end{remark}

\begin{lemma}\label{enum:conj-beta}
Conjugating $\A_n$ by ${\beta}_{n+1}$ gives $\widehat{\A}_n$.
\end{lemma}

\begin{proof} 
Conjugating $g\in \A_n$ by ${\beta}_{n+1}$ is same as applying the permutation ${\beta}_{n+1}$ to the labels of $g$; hence the resulting element is an element of $\widehat{\A}_n$ obtained by replacing $i$ with $2^n+i$ in the cycle notation of $g$. Therefore one has $\A_n{\beta}_{n+1}={\beta}_{n+1}\widehat{\A}_n$.
\end{proof}

\begin{example}
\label{ex:}
Consider the transposition $g=(1,2) \in \A_2$. Then 
\[
{\beta}_{3}^{-1} g {\beta}_{3} = ({\beta}_{3}(1), {\beta}_{4}(3)) = (5,6) \in \widehat{\A}_n
\]  
by Lemma~\ref{enum:conj-beta}. 
\end{example}

    Note that $\{{\beta}_i\}_{i=1,\dots,n}$ is a set of generators for $\A_n$. In terms of the tree diagrams, these elements correspond to swaps at leftmost nodes at every level of the tree. The automorphism ${\beta}_1$ corresponds to the top level, furthest away from the root. 
    
    We will also be referring to groups $\widehat{\A}_{n}$ and $\widehat{\A}_{n+1}$ as the subgroups of $\A_{n+2}$. 
    
    \input{fig3_1} 
    
    For example, in $T_4$, we have 
    \[ 
    \xymatrix@-1.6pc{
    *={\bullet} \save+<0ex,2ex>*={1} \restore
    & & 
    *={\bullet} \save+<0ex,2ex>*={2} \restore
    & & 
    *={\bullet} \save+<0ex,2ex>*={3} \restore
    & &
    *={\bullet} \save+<0ex,2ex>*={4} \restore
    & &
    *={\bullet} \save+<0ex,2ex>*={5} \restore
    & &
    *={\bullet} \save+<0ex,2ex>*={6} \restore
    & &
    *={\bullet} \save+<0ex,2ex>*={7} \restore
    & &
    *={\bullet} \save+<0ex,2ex>*={8} \restore 
    & &
    *={\bullet} \save+<0ex,2ex>*={9} \restore
    & & 
    *={\bullet} \save+<0ex,2ex>*={10} \restore
    & & 
    *={\bullet} \save+<0ex,2ex>*={11} \restore
    & &
    *={\bullet} \save+<0ex,2ex>*={12} \restore
    & &
    *={\bullet} \save+<0ex,2ex>*={13} \restore
    & &
    *={\bullet} \save+<0ex,2ex>*={14} \restore
    & &
    *={\bullet} \save+<0ex,2ex>*={15} \restore
    & &
    *={\bullet} \save+<0ex,2ex>*={16} \restore \frm{} \\
    &   \ar@{--}[ul] 
    *={\bullet} \save+<0ex,0ex>*={\bullet} \restore
    \ar@{--}[ur]
    & & & & \ar@{--}[ul] 
    *={\bullet} \save+<0ex,0ex>*={\bullet} \restore
    \ar@{--}[ur]  & & & & \ar@{=}[ul] 
    *={\bullet} \save+<0ex,0ex>*={\bullet} \restore
    \ar@{=}[ur] & & & & \ar@{=}[ul] 
    *={\bullet} \save+<0ex,0ex>*={\bullet} \restore
    \ar@{=}[ur] & 
    & &
    &   \ar@{~}[ul] 
    *={\bullet} \save+<0ex,0ex>*={\bullet} \restore
    \ar@{~}[ur]
    & & & & \ar@{~}[ul] 
    *={\bullet} \save+<0ex,0ex>*={\bullet} \restore
    \ar@{~}[ur]  & & & & \ar@{~}[ul] 
    *={\bullet} \save+<0ex,0ex>*={\bullet} \restore
    \ar@{~}[ur] & & & & \ar@{~}[ul] 
    *={\bullet} \save+<0ex,0ex>*={\bullet} \restore
    \ar@{~}[ur] & 
    \frm{}  \\
    & & & \ar@{--}[ull] 
    *={\bullet} \save+<0ex,0ex>*={\bullet} \restore
    \ar@{--}[urr] 
    \ar@{<..>}@/^0.10pc/[rrrrrrrr]^{\mbox{\large ${\beta}_3$}}
    & & & & & & & & 
    \ar@{=}[ull]
    *={\bullet} \save+<0ex,0ex>*={\bullet} \restore
    \ar@{=}[urr] 
    & & & 
    & & 
    & & & \ar@{~}[ull] 
    *={\bullet} \save+<0ex,0ex>*={\bullet} \restore
    \ar@{~}[urr] 
    \ar@{<..>}@/^0.10pc/[rrrrrrrr]^{\large \mbox{$\widehat{{\beta}}_3$}}
    & & & & & & & & 
    \ar@{~}[ull]
    *={\bullet} \save+<0ex,0ex>*={\bullet} \restore
    \ar@{~}[urr] 
    & & & \frm{} \\
    & & & & & & & 
    \ar@{<..>}@/^0.05pc/[rrrrrrrrrrrrrrrr]^{\mbox{\large ${\beta}_4$}}
    \ar@{-}[ullll] 
    *={\bullet} \save+<0ex,0ex>*={\bullet} \restore 
    \ar@{-}[urrrr]& & & & & & & 
    & & 
    & & & & & & & \ar@{~}[ullll] 
    *={\bullet} \save+<0ex,0ex>*={\bullet} \restore 
    \ar@{~}[urrrr]& & & & & & & 
    \frm{} \\ 
    & &  & &  & &  & &  & &
    & &  & &  & &  & &  & &
    & &  & &  & &  & &  & &  \\ 
    & &  & &  & &  & &  & &
    & &  & &  & 
    \ar@{-}[uullllllll]
    *={\circ} \save+<0ex,0ex>*={\circ} \restore
    \ar@{-}[uurrrrrrrr]
    & &  & &  & &  & &  & & 
    & &  & &  &
    },  \\ 
\] 
where 
\[ 
\mbox{\large $\A_2$}= \Aut\vast( \xymatrix@-1.25pc{
    *={\bullet} \save+<0ex,2ex>*={1} \restore 
   & &
    *={\bullet} \save+<0ex,2ex>*={2} \restore
   & & 
    *={\bullet} \save+<0ex,2ex>*={3} \restore
   & &  
    *={\bullet} \save+<0ex,2ex>*={4} \restore 
    \\ 
    & 
    \ar@{--}[ul]
    *={\bullet} \save+<0ex,0ex>*={\bullet} \restore
    \ar@{--}[ur]
    & & & &
    \ar@{--}[ul]
    *={\bullet} \save+<0ex,0ex>*={\bullet} \restore
    \ar@{--}[ur]
    &  \\ 
    & & & \ar@{--}[ull] 
    *={\circ} \save+<0ex,0ex>*={\circ} \restore
    \ar@{--}[urr] 
    & & &  \\
} \vast), 
\hspace{1cm}
\mbox{\large $\widehat{\A}_2$} = \Aut\vast( 
\xymatrix@-1.25pc{
    *={\bullet} \save+<0ex,2ex>*={5} \restore 
   & &
    *={\bullet} \save+<0ex,2ex>*={6} \restore
   & & 
    *={\bullet} \save+<0ex,2ex>*={7} \restore
   & &  
    *={\bullet} \save+<0ex,2ex>*={8} \restore 
    \\ 
    & 
    \ar@{=}[ul]
    *={\bullet} \save+<0ex,0ex>*={\bullet} \restore
    \ar@{=}[ur]
    & & & &
    \ar@{=}[ul]
    *={\bullet} \save+<0ex,0ex>*={\bullet} \restore
    \ar@{=}[ur]
    &  \\ 
    & & & \ar@{=}[ull] 
    *={\circ} \save+<0ex,0ex>*={\circ} \restore
    \ar@{=}[urr] 
    & & &  \\
} \vast), 
\] 
and 
\[ 
\mbox{\large $\widehat{\A}_3$}=\Aut \Vast( \xymatrix@-1.05pc{
    *={\bullet} \save+<0ex,2ex>*={9} \restore
    & & 
    *={\bullet} \save+<0ex,2ex>*={10} \restore
    & & 
    *={\bullet} \save+<0ex,2ex>*={11} \restore
    & &
    *={\bullet} \save+<0ex,2ex>*={12} \restore
    & &
    *={\bullet} \save+<0ex,2ex>*={13} \restore
    & &
    *={\bullet} \save+<0ex,2ex>*={14} \restore
    & &
    *={\bullet} \save+<0ex,2ex>*={15} \restore
    & &
    *={\bullet} \save+<0ex,2ex>*={16} \restore \frm{} \\
    &   \ar@{~}[ul] 
    *={\bullet} \save+<0ex,0ex>*={\bullet} \restore
    \ar@{~}[ur]
    & & & & \ar@{~}[ul] 
    *={\bullet} \save+<0ex,0ex>*={\bullet} \restore
    \ar@{~}[ur]  & & & & \ar@{~}[ul] 
    *={\bullet} \save+<0ex,0ex>*={\bullet} \restore
    \ar@{~}[ur] & & & & \ar@{~}[ul] 
    *={\bullet} \save+<0ex,0ex>*={\bullet} \restore
    \ar@{~}[ur] & 
    \frm{}  \\
    & & & \ar@{~}[ull] 
    *={\bullet} \save+<0ex,0ex>*={\bullet} \restore
    \ar@{~}[urr] 
    & & & & & & & & 
    \ar@{~}[ull]
    *={\bullet} \save+<0ex,0ex>*={\bullet} \restore
    \ar@{~}[urr] 
    & & & \frm{} \\ 
    & & & & & & & \ar@{~}[ullll] 
    *={\circ} \save+<0ex,0ex>*={\circ} \restore 
    \ar@{~}[urrrr]& & & & & & & 
    \frm{} \\ 
    } \Vast).
\] 
    
    \label{subsection:dot-diagram}

\section{Structural properties of the iterated wreath products of \texorpdfstring{$S_2$}{S2}}\label{section:structural-properties}

This section provides results about the centers, centralizers, cosets and their representatives for the groups $\A_n$, which will be needed in the rest of the paper.

    Recall that ${\beta}_i$ is the permutation on the leftmost vertex at the level $i$ of the binary tree $T_n$.  We start by citing a generators and relations description of iterated wreath products of symmetric groups.
    
    \begin{theorem}[Theorem 1.2, \cite{OOR}] 
    \label{thm:OOR-presentation}
    The group $\A_n$ is given by generators ${\beta}_1,\ldots, {\beta}_n$ with relations 
    \begin{enumerate}
    \item ${\beta}_i^2=\id$ for $1\leq i\leq n$, 
    \item $({\beta}_i{\beta}_j)^4=\id$ for $i\not=j$, $1\leq i,j\leq n$, and 
    \item $({\beta}_i{\beta}_j{\beta}_i{\beta}_{j+k})^2 = \id$ for $j>i$ and $1\leq k\leq n-i$. 
    \end{enumerate}
    \end{theorem}

\begin{proposition}{\label{prop:decomp-An}}
The group $\A_{n+l+1}$ decomposes as the disjoint union 
\[ \coprod_{x\in \mathcal{B}_{n,l+1} \cup \{ \id \}} 
   \A_{n} \widehat{\A}_n \widehat{\A}
_{n+1} \widehat{\A}_{n+2}\cdots \widehat{\A}_{n+l}x, \] 
where 
\begin{align*}
\mathcal{B}_{n,l+1} &:= \left\{ {\beta}_I := {\beta}_{i_1}{\beta}_{i_2}\cdots {\beta}_{i_s} | I= (i_1, i_2,\ldots, i_s) \right. \\ 
&\hspace{8mm} \left. \mbox{where }  n< i_1<i_2<\ldots<i_s \leq n+l+1 \right\}. 
\end{align*} 
\end{proposition}

\begin{proof}
We will prove by induction on $l$.

\textit{Case $l=0$}. This is easy to see when one considers $\A_{n+1}$ as the automorphism group of the truncated tree $T_{n+1}$. Such an automorphism either contains a swap at the root of the tree, \textit{i.e.}, ${\beta}_{n+1}$, or not. Automorphisms that do not contain ${\beta}_{n+1}$ consist of combinations of swaps at higher levels of the tree, so they are elements of $\A_n\widehat{\A}_n$. Note that since elements of $\A_n$ and $\widehat{\A}_n$ commute, this set is the same as $\widehat{\A}_n {\A_n}$. The remaining automorphisms should contain ${\beta}_{n+1}$ and may contain any other swap, so they are elements of $\A_n\widehat{\A}_n{\beta}_{n+1}$. Hence $\A_{n+1}=\A_n\widehat{\A}_n\sqcup \A_n\widehat{\A}_n{\beta}_{n+1}$.

Note that one can also express the set $\A_n\widehat{\A}_n{\beta}_{n+1}$ as ${\beta}_{n+1}\A_n\widehat{\A}_n = {\beta}_{n+1}\widehat{\A}_n \A_n = \A_n{\beta}_{n+1}\A_n=\widehat{\A}_n{\beta}_{n+1}\widehat{\A}_n$ using Lemma~\ref{enum:conj-beta}.

\textit{Inductive step}. Suppose the proposition holds for $\A_{n+l}$ for some $l$, that is $\A_{n+l}$ is the disjoint union 
\begin{equation}{\label{eq:decomposition}}
    \A_{n+l}= \coprod_{x\in \mathcal{B}_{n,l} \cup \{ \id \}} 
   \A_{n} \widehat{\A}_n \widehat{\A}_{n+1}\cdots \widehat{\A}_{n+l-1}x.
\end{equation}

Since $\A_{n+l+1} = \A_{n+l}\widehat{\A}_{n+l}\sqcup \A_{n+l}\widehat{\A}_{n+l}{\beta}_{n+l+1}$, by replacing $\A_{n+l}$ with the expression in \eqref{eq:decomposition} one gets 

\begin{align}
    \A_{n+l+1}&= \left(\coprod_{x\in \mathcal{B}_{n,l-1} \cup \{ \id \}} 
   \A_{n} \widehat{\A}_n \cdots \widehat{\A}_{n+l-1}x\right) \widehat{\A}_{n+l} \:\:\sqcup \\
   &\hspace{8mm}\left(\coprod_{x\in \mathcal{B}_{n,l-1} \cup \{ \id \}} 
   \A_{n} \widehat{\A}_n \cdots \widehat{\A}_{n+l-1}x\right) \widehat{\A}_{n+l}{\beta}_{n+l+1}\\
&= \left(\coprod_{x\in \mathcal{B}_{n,l-1} \cup \{ \id \}} 
   \A_{n} \widehat{\A}_n \cdots \widehat{\A}_{n+l-1} \widehat{\A}_{n+l}x\right) \:\: \sqcup \\ 
   &\hspace{8mm}\left(\coprod_{x\in \mathcal{B}_{n,l-1} \cup \{ \id \}} 
   \A_{n} \widehat{\A}_n \cdots \widehat{\A}_{n+l-1} \widehat{\A}_{n+l}x{\beta}_{n+l+1}\right) \\
   &=\coprod_{x\in \mathcal{B}_{n,l} \cup \{ \id \}} 
   \A_{n} \widehat{\A}_n \cdots \widehat{\A}_{n+l}x. 
\end{align}
This concludes the proof. 
\end{proof}

\begin{proposition}{\label{prop:rightcoset-n+l}}
The set of right cosets $\A_{n}\backslash  \A_{n+l+1}$ decomposes as the disjoint union 
   \[ 
   \A_{n}\backslash \A_{n+l+1} = \coprod_{x\in \mathcal{B}_{n,l+1} \cup \{ \id \}} 
   \A_{n}\backslash \left(\A_{n} \widehat{\A}_n \widehat{\A}_{n+1} \widehat{\A}_{n+2}\cdots \widehat{\A}_{n+l}x\right), 
\] 
where 
\begin{align*} 
\mathcal{B}_{n,l+1} &:= \left\{ {\beta}_I := {\beta}_{i_1}{\beta}_{i_2}\cdots {\beta}_{i_s} : I= (i_1, i_2,\ldots, i_s)  \right. \\ 
&\hspace{8mm} \left. \mbox{where }  n< i_1<i_2<\ldots<i_s \leq n+l+1 \right\}. 
\end{align*}
\end{proposition}

\begin{proof}
It is enough to show that the sets $\A_{n} \widehat{\A}_n \widehat{\A}_{n+1} \widehat{\A}_{n+2}\cdots \widehat{\A}_{n+l}x$ that partition $\A_{n+l+1}$ are invariant under left multiplication by $g\in \A_n$. But this obvious since $g\A_n=\A_n$ implies that  
\[ 
g\A_{n} \widehat{\A}_n \widehat{\A}_{n+1} \cdots \widehat{\A}_{n+l}x=\A_{n} \widehat{\A}_n \widehat{\A}_{n+1} \cdots \widehat{\A}_{n+l}x. 
\]
\end{proof}

   \begin{proposition}
   \label{prop:rightcoset-rep-gen}
   A set of representatives for the right cosets $\A_{n}\backslash \A_{n+l+1}$ is given by the disjoint union 
   $\displaystyle \coprod_{x\in \mathcal{B}_{n,l+1} \cup \{ \id \}} 
   \widehat{\A}_{n} \widehat{\A}_{n+1} \widehat{\A}_{n+2}\cdots \widehat{\A}_{n+l}x$.
   \end{proposition}
   
\begin{proof}
By Proposition~\ref{prop:rightcoset-n+l}, 
we have 
\begin{align*}
    \{\A_n g\}_{g\in \A_{n+l+1}}&=\coprod_{x\in \mathcal{B}_{n,l+1} \cup \{ \id \}} \{\A_n g\}_{g\in \A_n\widehat{\A}_{n} \widehat{\A}
_{n+1} \widehat{\A}_{n+2}\cdots \widehat{\A}_{n+l}x}\\
&=\coprod_{x\in \mathcal{B}_{n,l+1} \cup \{ \id \}}\{\A_n h\}_{h\in \widehat{\A}_{n} \widehat{\A}_{n+1} \widehat{\A}_{n+2}\cdots \widehat{\A}_{n+l}x} \\
&=\coprod_{x\in \mathcal{B}_{n,l+1} \cup \{ \id \}} 
\left(\coprod_{h\in \widehat{\A}_{n} \widehat{\A}_{n+1} \widehat{\A}_{n+2}\cdots \widehat{\A}_{n+l}x} \A_n h \right), 
\end{align*}
which completes the proof. 
\end{proof}

\begin{remark}{\label{rmk:rightcoset-decomp}} 
For $l=0$ Proposition \ref{prop:rightcoset-rep-gen} reduces to the statement that a set of representatives for the right cosets $\A_{n}\backslash \A_{n+1}$ is given by $\widehat{\A}_{n}\sqcup\widehat{\A}_{n}{\beta}_{n+1}$.
\end{remark}

\begin{lemma}{\label{center}}
The center of $\A_n$ is 
    \[
    Z(\A_n) = \{\id, (1,2)(3,4)\cdots(2^n-1,2^n)\}.
    \] 
\end{lemma}
\begin{proof}
    Let $\alpha \in Z(\A_n)$ and consider the generating set $B=\{ {\beta}_1,{\beta}_2,\dots,{\beta}_n \}$ of $\A_n$ where ${\beta}_k=\displaystyle{\prod_{i=1}^{2^k}}(i,2^k+i) $.  We know that $\alpha$ should satisfy $\alpha {\beta}_k \alpha^{-1}={\beta}_k$ for all ${\beta}_k\in B$. Conjugation by a permutation has the effect of relabeling, hence we get 
    \begin{equation}{\label{relabel}}
        \alpha {\beta}_k \alpha^{-1}= \prod_{i=1}^{2^k}(\alpha(i),\alpha(2^k+i))=\prod_{i=1}^{2^k}(i,2^k+i)={\beta}_k.
    \end{equation}
    This equation puts a restriction on $\alpha$ for each $k\in \{1,\dots, n\}$.
    
    For $k=1$, we get $(1,2)=(\alpha(1),\alpha(2))$, so either $\alpha$ fixes $1$ and $2$, or it swaps $1$ and $2$. 
    
    For $k=2$, we get $(1,3)(2,4)=(\alpha(1),\alpha(3))(\alpha(2),\alpha(4))$, so if $\alpha$ fixes $1,2$, we must have $(1,3)(2,4)=(1,\alpha(3))(2,\alpha(4))$, which forces $\alpha$ to fix $3$ and $4$ as well. Otherwise, if it swaps $1$ and $2$, we obtain $(1,3)(2,4)=(2,\alpha(3))(1,\alpha(4))$, which forces $\alpha$ to swap $3$ and $4$ as well. 
    
    Since Equation~\eqref{relabel} holds for $k=0,1,\dots,n-1$, repeating similar arguments, we get that either $\alpha=\id$, or $\alpha=(1,2)(3,4)\cdots(2^n-1,2^n)$.
\end{proof}

Define the centralizer of $\A_n$ in $\A_{n+k}$ as 
\begin{equation}
\label{eqn:centralizer-group}
Z_{n,k}:=Z(\A_{n+k},\A_n)=\left\{g\in \A_{n+k}: gh = hg    \mbox{ for all }  h\in \A_n\right\}.      
\end{equation}
We note that the usual notation in group theory for the centralizer is $C_{\A_{n+k}}(\A_n)$.

\begin{lemma}{\label{centralizer}}
   The centralizer $Z_{n,1}$ of $\A_n$ in $\A_{n+1}$ is equal to $Z(\A_n)\widehat{\A}_n$.
\end{lemma}

\begin{proof}
    Again, it will suffice to work with the generators $B=\{ {\beta}_1,{\beta}_2,\dots,{\beta}_n\}$ of $\A_n$. 
    Let $\alpha=ab$, where $a\in Z(\A_n)$ and $b\in \widehat{\A}_n$. Then $\alpha {\beta}_k = ab {\beta}_k = a {\beta}_k b = {\beta}_k ab = {\beta}_k \alpha$, 
    where the second equality holds since $\widehat{\A}_n$ commutes with $\A_n$ (see Lemma~\ref{enum:labelsets-L1L2}) and the third equality holds since $a$ is in the center of $\A_n$. 
     
    On the other hand, for $\alpha \in Z_{n,1}$, the condition $\alpha {\beta}_k \alpha^{-1}={\beta}_k$ for $k\in\{1,\dots,n\}$ forces $\alpha$ to either fix $1,2,\dots,2^n$, or to swap $1$ and $2$, $3$ and $4$,$\dots$, $2^{n}-1$ and $2^n$, and there are no restrictions on the labels $\{2^n+1,\dots,2^{n+1}\}$. Since $\widehat{\A}_n$ is the subgroup of $\A_{n+1}$ that permutes these labels in any possible way, $\alpha$ is of the form $ab$, where $a\in Z(\A_n)$ and $b \in \widehat{\A}_n$. Therefore $Z_{n,1}=Z(\A_n)\widehat{\A}_n$.
\end{proof}    

\begin{lemma}{\label{bigdoublecoset}}
$\A_n {\beta}_{n+1} \A_n$ is an $(\A_n,\A_n)$-double coset of $\A_{n+1}$ with $|\A_n|^2$ elements.
\end{lemma}

\begin{proof}
    We want to show that the elements $x{\beta}_{n+1} y$ are all distinct for $x,y\in \A_n$. 
    So suppose $x{\beta}_{n+1} y = x'{\beta}_{n+1} y'$, where $x,y,x',y'\in \A_n$. 
    Then $(x')^{-1}x={\beta}_{n+1}y'y^{-1}{\beta}_{n+1}^{-1}$. However, conjugation by ${\beta}_{n+1}$ has the effect of sending labels $\{1,2,\dots,2^n\}$ to $\{2^n+1,\dots,2^{n+1}\}$. 
    So the permutations $(x')^{-1}x$ and ${\beta}_{n+1}y'y^{-1}{\beta}_{n+1}^{-1}$ are acting on disjoint sets in a partition of all the labels for $\A_{n+1}$. So $(x')^{-1}x={\beta}_{n+1}y'y^{-1}{\beta}_{n+1}^{-1}=\id$ implies $x=x'$ and $y=y'$. This proves that $x{\beta}_{n+1} y$ are all distinct for $x,y\in \A_n$. Hence $\A_n {\beta}_{n+1} \A_n$ is a $(\A_n,\A_n)$-double coset of $\A_{n+1}$ with $|\A_n|^2$ elements.
\end{proof}

\begin{proposition}{\label{doublecosets}}
The elements $\widehat{\A}_n\cup\{{\beta}_{n+1}\}$ form a  complete set of representatives for $(\A_n,\A_n)$-double cosets of $\A_{n+1}$.
\end{proposition}

\begin{proof} 
We will provide a proof by showing that each of these representatives gives disjoint double cosets and by using a counting argument, \textit{i.e.}, their total cardinality is equal to the cardinality of $\A_{n+1}$.

For $b \in \widehat{\A}_n\subset Z_{n,1}$, we get $\A_n b \A_n=b \A_n$ since elements in $\A_n$ and $\widehat{\A}_n$ commute. This gives us $|\A_n|$ disjoint double cosets (they are disjoint because each $b$ corresponds to a different swap at the right branch of the tree $T_{n+1}$), each with $|\A_n|$ elements. So the double cosets $\A_nb\A_n$ for $b\in \widehat{\A}_n$ count for $|\A_n||\A_n|$ many elements of $\A_{n+1}$.

Since the cardinality of $\A_{n+1}$ is $|\A_{n+1}| = |\A_n \wr S_2| = |(\A_n \times \A_n)\rtimes S_2| 
= 2|\A_n|^2$, there are still $2|\A_n|^2-|\A_n||\A_n|$ elements we have to account for.

    Note that $\A_{n+1} 
    = \A_n \widehat{\A}_n \A_n \sqcup \A_n{\beta} \A_n
    = \A_n \widehat{\A}_n \sqcup \A_n{\beta} \A_n$.
    It follows from 
    Lemma~\ref{centralizer} that each element of $\widehat{\A}_n$ gives a distinct double coset of size $|\A_n|$, and it follows from Lemma~\ref{bigdoublecoset} that ${\beta}_{n+1}$ gives a double coset of size $|\A_n|^2$. This counts for all the elements of $\A_{n+1}$.
\end{proof}

\begin{proposition}
\label{prop:basis-orbit-sum}
A basis of the subalgebra $\GZ{n,k}$ is given by the orbit sums of the conjugation action $\A_{n} \curvearrowright \A_{n+k}$.
\end{proposition}

This is an application of the classical yet extremely powerful idea of averaging over group elements ubiquitous in representation theory. We will now prove Proposition~\ref{prop:basis-orbit-sum}.

\begin{proof}
    Let $\orbit$ be an orbit and let $\displaystyle v_{\orbit}=\sum_{g\in \orbit}g$. Then for any $a\in \A_{n}$, we have 
    \begin{equation*}
        av_{\orbit} = a\left(\sum_{g\in \orbit}g\right)
        =\sum_{g\in \orbit}ag
        =\sum_{g'\in \orbit}g'a
        =\left(\sum_{g'\in \orbit}g'\right)a=v_{\orbit}a.
    \end{equation*}
    Hence $v_{\orbit} \in \GZ{n,k}$.
    
    Conversely, let $v=c_1v_{h_1}+\ldots+c_kv_{h_k}\in Z_{n,k}$ where $\{v_{h_i}\}_{h_i\in \A_{n+k}}$ is a basis of $\GA{n+1}$ and suppose that $vv_{g}=v_{g}v$ for all $g\in i(\A_{n})$. Then we have \begin{equation*}
        \sum_{i=1}^kc_iv_{h_i}=v=v_{g}vv_{g^{-1}}=\sum_{i=1}^kc_iv_{g}v_{h_i}v_{g^{-1}}=\sum_{i=1}^kc_iv_{gh_ig^{-1}}
    \end{equation*}
    Since this is true for all $g\in \A_{n}$, all the basis vectors $v_{h_j}$ for $h_j$ in the orbit of $h_i$ should be present in $v$ for all $i=1,2,\ldots,k$. Moreover in order for the above equation to hold, one needs the coefficients of the basis vectors indexed by elements from the same orbit to be equal to each other. Therefore $\{v_{\orbit}\}_{\orbit\text{ orbit}}$ is a basis of $\GZ{n,k}$.
\end{proof}

There is a special element whose orbit exhibits a  notable behaviour. The orbit of ${\beta}_{n+1}$ is the same under the conjugation actions of $\A_n$ and $\A_{n+1}$.
\begin{proposition}
\label{prop:orbitn-orbitnplus1}
Let  $\orbit_{n}({\beta}_{n+1}):=\{h{\beta}_{n+1}h^{-1}: h\in \A_n\}$ be the orbit of ${\beta}_{n+1}$ under the conjugation action of $\A_n$ on $\A_{n+1}$, and let $\orbit_{n+1}({\beta}_{n+1}) :=\{g{\beta}_{n+1}g^{-1}: g\in \A_{n+1}\}$ be the orbit of ${\beta}_{n+1}$ under the conjugation action of $\A_{n+1}$ on $\A_{n+1}$, \textit{i.e.}, the conjugacy class of ${\beta}_{n+1}$ in $\A_{n+1}$. Then $\orbit_{n}({\beta}_{n+1})= \orbit_{n+1}({\beta}_{n+1})$.
\end{proposition}

\begin{proof}
    Note that ${\beta}_{n+1}$ is a product of transpositions of the form $(\ell_1,\ell_2)$ where $\ell_1\in L_1=\{1,\dots,2^n\}$, $\ell_2\in L_2=\{2^n+1,\dots,2^{n+1}\}$. Conjugation by $h\in \A_n$ has the effect of permuting only the labels from the set $L_1$.
    
    Conjugation by $g\in \A_{n+1}$ affects both the labels $\ell_1$ and $\ell_2$. Every $g$ can be written as $h{\beta}_{n+1}k$ for $h\in \A_n$ and $k\in \widehat{\A}_n$. Here only ${\beta}_{n+1}$ can exchange labels between $L_1$ and $L_2$, however note that no matter which element of $\A_{n+1}$ we choose, we can never obtain a permutation containing $(\ell_1,\ell_1')$ since if ${\beta}_{n+1}(\ell_2)$ is in $L_1$, then ${\beta}_{n+1}(\ell_1)$ should be in $L_2$.
    
    Therefore conjugates of ${\beta}_{n+1}$ are products of transpositions $(\ell_1,\ell_2)$, which can all be obtained by conjugation with $h\in \A_n$. Hence $\orbit_{n}({\beta}_{n+1})= \orbit_{n+1}({\beta}_{n+1})$.
\end{proof}

\begin{remark}{\label{rmk:orbitsum_beta_n_central}}
Let $\orbitsum_n(g)$ denote the sum of the elements in the orbit $\orbit_n(g)$, \textit{i.e.}, $\orbitsum_n(g) = \displaystyle{\sum_{g\in \orbit_n(g)}}g$. 
Note that $\orbitsum_{n+1}({\beta}_{n+1})$ is a central element of $\GA{n+1}$ since it is a conjugacy class sum. As a result, we get that $\orbitsum_{n}({\beta}_{n+1})$ is also central in $\GA{n+1}$.
\end{remark}

Now we are in a position to give an explicit vector space basis of $\GZ{n,1}$.

    \begin{proposition}{\label{orbits}}
    The number of orbits of the conjugation action $\A_{n} \curvearrowright \A_{n+1}$ is $|\A_{n}|\times (c_n+1)$ where $c_n$ is the number of conjugacy classes of $\A_n$. Moreover, if we denote by $C_n(g)$ the conjugacy class of $g\in \A_n$, then the orbits are $\left\{C_n(g)h\right\}_{g\in \A_n,h \in \widehat{\A}_n}$  and $\left\{\orbit_n({\beta}_{n+1})h\right\}_{ h\in \widehat{\A}_n}$. 
    \end{proposition}
    
    \begin{proof}
        We will use the decomposition $\A_{n+1} = \A_n\widehat{\A}_n\sqcup  \A_n\widehat{\A}_n{\beta}_{n+1} $. Since $\A_n$ and $\widehat{\A}_n$ commute, the action of $\A_n$ on $\A_n\widehat{\A}_n$ results in orbits of the form $C_n(a)h$ for $a\in \A_n, h\in \widehat{\A}_n$ where $C_n(a)$ is the conjugacy class of $a$ in $\A_n$. 
        
        As for the action on $\A_n\widehat{\A}_n{\beta}_{n+1}$, note that we have the equalities
        \begin{equation*}
            \A_n\widehat{\A}_n{\beta}_{n+1}=\A_n{\beta}_{n+1}\A_n={\beta}_{n+1}\widehat{\A}_n\A_n={\beta}_{n+1}\A_n\widehat{\A}_n=\widehat{\A}_n{\beta}_{n+1}\widehat{\A}_n
        \end{equation*}
        
        Therefore if we choose to work with $\widehat{\A}_n{\beta}_{n+1}\widehat{\A}_n$, for $g\in \A_n$ and $\widehat{h},\widehat{k} \in \widehat{\A}_n$, \begin{equation*}
            g\widehat{h}{\beta}_{n+1}\widehat{k}g^{-1}=\widehat{h}g{\beta}_{n+1}g^{-1}\widehat{k}.
        \end{equation*} 
        Hence we get orbits $\widehat{h}\orbit_n({\beta}_{n+1})\widehat{k}$. Since $\orbit_n({\beta}_{n+1})$ is equal to the conjugacy class $\orbit_{n+1}({\beta}_{n+1})$ in $\A_{n+1}$ by Proposition \ref{prop:orbitn-orbitnplus1} and any $\widehat{k}\in \widehat{\A}_n$ can be expressed as $\widehat{h}^{-1}\widehat{\ell}_{k,h}$, we get
        $$
        \widehat{h}\orbit_n({\beta}_{n+1})\widehat{k}= \widehat{h}\orbit_n({\beta}_{n+1})\widehat{h}^{-1}\widehat{\ell}_{k,h}=\orbit_n({\beta}_{n+1})\widehat{\ell}_{k,h}.
        $$
        Therefore, the orbits are of the form $\orbit_n({\beta}_{n+1})h$ for $h\in \widehat{\A}_n$.
    \end{proof}
    
    This determines a basis for $\GZ{n,1}$ and its dimension is $|\A_n|(c_n+1)$. To get an explicit formula for this dimension, we need an explicit formula for the number of conjugacy classes of $\A_n$, which is same as the number of irreducible representations of $\A_n$.

    We restrict \cite[Theorem 2.2]{OOR} to $r=2$ to obtain the number of conjugacy classes: 
    \begin{theorem}[Orellana--Orrison--Rockmore]
    \label{theorem:OOR}
     The number $c_n$ of irreducible representations of $\A_n$ is given by $c_0=1$ and the recursion 
     \[ 
     c_n=\frac{1}{2}c_{n-1}(3+c_{n-1}). 
     \]
    \end{theorem}
    
    Note that the sequence $\{c_n\}_{n\in \mathbb{N}}$ appears in OEIS as sequence $\href{https://oeis.org/search?q=A006893&language=english}{A006893}$.

    The following proposition describes the orbits of the action $\A_n\curvearrowright \A_{n+k+1}$. An explicit vector space basis for $\GZ{n,k+1}$ is given by their orbit sums.
    
    \begin{proposition}
    \label{prop:conj-class-Indk}
    The number of orbits of the conjugation action $\A_{n} \curvearrowright \A_{n+k+1}$ is $|\A_{n}||\A_{n+1}|\cdots |\A_{n+k}| (c_n+a_k)$ where $c_n$ is the number of conjugacy classes of $\A_n$, $a_k$ is the sequence defined recursively as $a_0=0$, $a_r=2a_{r-1}+1$ for $r>0$.

    Moreover, if we denote by $c_n(g)$ the conjugacy class sum of $g$ in $\A_n$, then the orbits are
    \begin{enumerate}
        \item $\left\{c_n(g)h: h \in \widehat{\A}_n\widehat{\A}_{n+1} \cdots \widehat{\A}_{n+k}\right\}$ for $g\in \A_n$ and
        \item $\left\{\orbit_n({\beta}_I)h : h\in \widehat{\A}_n\widehat{\A}_{n+1} \cdots \widehat{\A}_{n+k}\right\}$,
        \end{enumerate}
        where ${\beta}_{I}$ are from the set $\mathcal{B}_{n,k+1}\cup \{\id\}$.
    \end{proposition}

    \begin{proof}
        We will give a proof by induction on $k$. The base case $k=0$ is proved in Proposition~\ref{orbits}. Now assume the proposition holds for some $k$. We will make use of the decomposition $\A_{n+k+1} = \A_{n+k}\widehat{\A}_{n+k}\sqcup {\beta}_{n+k+1} \A_{n+k}\widehat{\A}_{n+k}$. This decomposition allows us to describe the orbits by studying the action of $\A_n$ on $\A_{n+k}\widehat{\A}_{n+k}$ and on ${\beta}_{n+k+1}\A_{n+k}\widehat{\A}_{n+k}$ separately.
        
        For the action of $\A_n$ on $\A_{n+k}\widehat{\A}_{n+k}$, since $\A_n$ and $\widehat{\A}_{n+k}$ commute, for $g\in \A_n$, $h\in \A_{n+k}$, $j\in \widehat{\A}_{n+k}$, we get $ghjg^{-1}=ghg^{-1}j$. Therefore the orbits of the action of $\A_n$ on $\A_{n+k}\widehat{\A}_{n+k}$ are given by orbits of $\A_n$ on $\A_{n+k}$ times an element of $\widehat{\A}_{n+k}$.
        
        Now for the action $\A_n\curvearrowright {\beta}_{n+k+1}\A_{n+k}\widehat{\A}_{n+k}$, since we have 
        $$
        \A_{n+k} = \A_{n+k-1}\widehat{\A}_{n+k-1} \sqcup {\beta}_{n+k}\A_{n+k-1}\widehat{\A}_{n+k-1}
        $$ 
        using this relation inductively one ends up with 
        \[
        \A_{n+k}=\bigsqcup_{{\beta}_I \in \mathcal{B}_{n,k} \cup \{\id\}}{\beta}_I \A_n\widehat{\A}_n\widehat{\A}_{n+1} \cdots \widehat{\A}_{n+k-1}
        \] 
        and obtains 
        \begin{align*}
            {\beta}_{n+k+1}\A_{n+k}\widehat{\A}_{n+k} &= \bigsqcup_{{\beta}_I \in \mathcal{B}_{n,k} \cup \{\id\}}{\beta}_{n+k+1}{\beta}_I \A_n\widehat{\A}_n\widehat{\A}_{n+1} \cdots \widehat{\A}_{n+k-1}\widehat{\A}_{n+k}\\
            &=\bigsqcup_{{\beta}_I \in \mathcal{B}_{n,k+1} \cup \{\id\}}{\beta}_I \A_n\widehat{\A}_n\widehat{\A}_{n+1} \cdots \widehat{\A}_{n+k}. 
        \end{align*}
        
        For $g,r\in \A_n$,$\widehat{m} \in \widehat{\A}_n\widehat{\A}_{n+1} \cdots \widehat{\A}_{n+k}$, we want to study $g{\beta}_Ir\widehat{m}g^{-1}$. Note that
         it is possible to express $\widehat{m}$ as $\widehat{r}^{-1}\widehat{\ell}_{r,m}$, where ${\beta}_I\widehat{r}^{-1}=r^{-1}{\beta}_I$. Hence we get
        \begin{align*}
            g{\beta}_Ir\widehat{m}g^{-1} &= g{\beta}_Ir\widehat{r}^{-1}\widehat{\ell}_{r,m}g^{-1} \\ &=g{\beta}_I\widehat{r}^{-1}r\widehat{\ell}_{r,m}g^{-1}  \\
            &=gr^{-1}{\beta}_Ir\widehat{\ell}_{r,m}g^{-1}  \\
            &=gr^{-1}{\beta}_Irg^{-1} \widehat{\ell}_{r,m}, 
        \end{align*}
        ending up with orbits of the form $\orbit_n{{\beta}_I}h$, where ${\beta}_I$ $\in$ $\mathcal{B}_{n,k+1} \cup \{\id\}$ and $h \in \widehat{\A}_n\widehat{\A}_{n+1} \cdots \widehat{\A}_{n+k}$.
    \end{proof}

\section{Natural transformation between induction and restriction}
\label{section:natural-transformation-ind-res}

Let the symbol $_n(k)_m$ denote the $\GA{k}$ as an $(\A_n,\A_m)$-bimodule.

Objects of the category $\mathcal{\A}_n = \displaystyle{\bigoplus_{m\in \mathbb{N}}}(\A_m,\A_n)$-$\bimod$ are generated by  $(\A_m,\A_n)$-bimodules for a fixed $n$. Since $\Ind_n^{n+1}$ and $\Res_n^{n-1}$ correspond to  $_{n+1}(n+1)_{n}\otimes_n-$ and  $_{n-1}(n-1)_{n}\otimes_n-$, respectively, one can identify the bimodules $_{n+1}(n+1)_{n}$ and $_{n-1}(n-1)_{n}$ with $\Ind_n^{n+1}$ and $\Res_n^{n-1}$, respectively. Therefore it is possible to describe the category $\mathcal{\A}_n$ in terms of compositions of $\Ind$ and $\Res$ and natural transformations between them as well. Let us denote by $\mathcal{C}_n$ the category whose objects are generated by compositions of induction $\Ind$ and restriction $\Res$ functors between groups $\{\A_k\}_{k\in \mathbb{N}}$, where the compositions start from induction or restriction from $\A_n$. 

Given a sequence of inductions and restrictions, using the formula \eqref{eqn:ResInd} below, it is possible to move inductions to the left and push restrictions to the right. Hence objects of $\mathcal{C}_n$ are given by direct sums of $\Ind^k\Res^l$ starting from $(n,n)$-bimodules. 

Morphisms of $\mathcal{C}_n$ consist of sets $\Hom_{\mathcal{C}_n}(\Ind^{k_1}\Res^{l_1}, \Ind^{k_2}\Res^{l_2})$. Since we start by applying the restrictions to an $(n,n)$-bimodule, it is clear that these hom spaces are trivial when $l_1>n$ or $l_2>n$. Also note that when $k_1-l_1\neq k_2-l_2$, we are talking about bimodule maps between two bimodules over different rings. Therefore in this case the hom space is trivial as well. 

From now on, we will assume that $k_1-l_1=k_2-l_2$, $l_1\leq n$ and $l_2\leq n$. We would like to describe the spaces $\Hom_{\mathcal{A}_n}({\Ind^{k_1}\Res^{l_1}, \Ind^{k_2}\Res^{l_2}})$, or $\Hom_n(+^{k_1}-^{l_1},+^{k_2}-^{l_2})$ for short, as a vector space and as an algebra.

In this section, we describe all these spaces implicitly and provide explicit descriptions for the cases $k_1=k_2=k$ (hence $l_1=l_2=l$) and $l=0,1$.

\subsection{Mackey Theorem}\label{subsection:Mackey-thm}
Let us start by recalling the general form of Mackey theorem for induction and restriction.

\begin{theorem}[Mackey Theorem]
    Let $G$ be a finite group and $H,K$ be two subgroups of $G$. Then \begin{equation}
        \Res_G^K \circ \Ind_H^G = \bigoplus_{g_i\in I} \Ind_{K\cap H^{g_i}}^K \circ \Res_H^{K\cap H^{g_i}}, 
    \end{equation}
    where $I$ is a set of $(H,K)$-double coset representatives of $G$ and $H^g := gHg^{-1}$.
\end{theorem}

In our case by letting $G=\A_{n+1}$ and $H=K=\A_n$, we get \begin{equation}
\label{eqn:Mackey-ResInd}
    \Res_{\A_{n+1}}^{\A_n} \circ \Ind_{\A_n}^{\A_{n+1}} = \bigoplus_{g_i\in I} \Ind_{\A_n \cap \A_n^{g_i}}^{\A_n} \circ \Res_{\A_n}^{\A_n\cap \A_n^{g_i}}, 
\end{equation}
where $I = \A_n\reflectbox{/} \A_{n+1}/\A_n$. 

Using Proposition~\ref{doublecosets} which informs us about double cosets, we can express the above formula more explicitly once we say something about $\A_n \cap \A_n^g$ for double coset representatives $g$. We will work with the double coset representatives $\{b\in \widehat{\A}_n\}\cup\{{\beta}_{n+1}\}$.
    
    For $b\in \widehat{\A}_n$, $\A_n^b=b \A_n b^{-1}=\A_n$ since  $\A_n$ and $\widehat{\A}_n$ commute. Hence the intersection of $\A_n$ and $\A_n^b$ is again $\A_n$. 
    
    For ${\beta}_{n+1}$, we have  $\A_n^{{\beta}_{n+1}}={\beta}_{n+1}\A_n{\beta}_{n+1}=\widehat{\A}_n$, hence $\A_n\cap \A_n^{{\beta}_{n+1}}=\{\id\} = \A_0$. Therefore the Mackey formula for the groups $\A_n$ gives
    \begin{equation}
    \label{eqn:ResInd}
        \Res_{\A_{n+1}}^{\A_n} \circ \Ind_{\A_n}^{\A_{n+1}} = \Ind_{\A_0}^{\A_n} \circ \Res_{\A_n}^{\A_0} \oplus \Id^{\oplus|\A_n|}. 
    \end{equation}
The fact that the number of summands of $\Id$ on the right hand side of \eqref{eqn:ResInd} depends on $n$ suggests that this line of investigation should be different than the Heisenberg categories in literature. More precisely, there should not exist an abstract 1-category $\mathcal{C}$ encapsulating the induction and restriction functors on the tower $\A_0\longhookrightarrow \A_1\longhookrightarrow \A_2\longhookrightarrow \dots$ and natural transformation between them which will act on the categories 
$\mathcal{\A}_n = \displaystyle{\bigoplus_{m\in \mathbb{N}}} (\A_m,\A_n)$-$\bimod$ for all $n$ as bimodule homomorphisms. Therefore we will focus on describing the explicit categories $\mathcal{\A}_n = \displaystyle{\bigoplus_{m\in \mathbb{N}}}(\A_m,\A_n)$-$\bimod$.

Note that given the induction functor 
\begin{center}
$\Ind_n^{n+1}:(\GA{n},\GA{n})$-$\bimod$ $\lra$ $(\GA{n+1},\GA{n})$-$\bimod$, 
\end{center}
we want to describe the natural transformations $\End_n(\Ind_n^{n+1})$ such that the diagram 
\[ 
\xymatrix@-1pc{
\Ind_n^{n+1}(\GA{n}) \ar[rrr]^{\Ind_n^{n+1}(h)} \ar[dd]_g & & & \Ind_n^{n+1}(\GA{n}) \ar[dd]^f \\ 
& & &  \\ 
\Ind_n^{n+1}(\GA{n}) \ar[rrr]_{\Ind_n^{n+1}(h)} & & & \Ind_n^{n+1}(\GA{n}) \\ 
}
\]  
commutes, where $h:\GA{n} \lra \GA{n}$ is an $(\GA{n},\GA{n})$-bimodule homomorphism. 
Suppose $f(\id) = a\in \GA{n+1}$. Then for any $\alpha\in \GA{n+1}$ and ${\beta}\in \GA{n}$, 
$f(\alpha x {\beta}) = \alpha f(x) {\beta} = \alpha xa {\beta}$, where $f(x)$ is multiplication by $a$ from the right. On the other hand, $f(\alpha x {\beta})=\alpha x{\beta} a$. So ${\beta} a=a{\beta} $ for any ${\beta} \in \GA{n}$ implies $f\in \End(\Ind_n^{n+1})$ if and only if the bimodule homomorphism $f$ is multiplication by an element from the centralizer $Z_{n,1}$.

    \subsection{Morphism spaces as vector spaces} 
    In this subsection, we give a description of the morphism spaces between induction and restriction functors on the representations of the tower 
    \[ 
    \A_1\stackrel{i_1}{\longhookrightarrow} \A_2 \stackrel{i_2}{\longhookrightarrow} \A_3 
    \stackrel{i_3}{\longhookrightarrow} \ldots.
    \] 
    as vector spaces. 
    
    Unlike the case for symmetric groups (cf. \cite[Thm.~2.1]{VO-symmetric-groups}), the centralizer for our tower of group algebras forms a noncommutative algebra. The size of our groups grows exponentially and as a consequence $\Res$ is not multiplicity free for this tower. 
    
    $\End_{\mathcal{C}_n}(\Ind^k\Res^l)$ consists of natural transformation satisfying the following diagram:
    \[ 
\xymatrix@-.5pc{
\Ind_n^{n+k}\Res_n^{n-l}(\GA{n}) \ar[rrrr]^{\Ind_n^{n+k}\Res_n^{n-l}(h)} \ar[dd]_g & & & & \Ind_n^{n+k}\Res_n^{n-l}(\GA{n}) \ar[dd]^f \\ 
& & & & \\ 
\Ind_n^{n+k}\Res_n^{n-l}(\GA{n}) \ar[rrrr]_{\Ind_n^{n+k}\Res_n^{n-l}(h)} & & & & \Ind_n^{n+k}\Res_n^{n-l}(\GA{n}) \\ 
}
\]  
In terms of bimodules, this amounts to 
    \[ 
\xymatrix@-.5pc{
_{n+r}(n+r)_{n-l}(n)_n \ar[rrrr]^{\Ind_n^{n+k}\Res_n^{n-l}(h)} \ar[dd]_g & & & & _{n+r}(n+r)_{n-l}(n)_n \ar[dd]^f \\ 
& & & & \\ 
_{n+r}(n+r)_{n-l}(n)_n \ar[rrrr]_{\Ind_n^{n+k}\Res_n^{n-l}(h)} & & & & _{n+r}(n+r)_{n-l}(n)_n \\ 
}
\]  
where $r=k-l$. Therefore if $f(1\otimes_{n-l}1)=y\in \GA{n+r}\otimes_{n-l} \GA{n}$, for $a\in \GA{n+r}, b\in \GA{n-l}$, we require that $$ayb=af(1\otimes_{n-l}1)b=f(a\otimes_{n-l}b)=f(ab\otimes_{n-l}1)=aby.$$ 

This calculation shows that 
$$
\End_n(\Ind^k\Res^l)=\{x \in \GA{n+r}\otimes_{n-l} \GA{n} : gx=xg \:\:\mbox{ for all } g\in \A_{n-l} \}. 
$$

For $l=0$, we get that $$
\End_n(\Ind^k)=Z(\GA{n+k}, \GA{n})
$$, 
which is the centralizer of the image of $\A_n$ inside $\GA{n+k}$. Let us denote it with $\GZ{n,k}$. A vector space basis of $\GZ{n,k}$, hence of $\End_n(\Ind^k)$ is given explicitly by orbit sums described in Proposition~\ref{prop:conj-class-Indk}.

Using the biadjointness between $\Ind$ and $\Res$, one immediately obtains a similar result for $\End_n(\Res^k)$.

    \begin{lemma}\label{lemma:conj-class-Resk}
    $\End_n(\Res^k)$ is isomorphic to $\End_n(\Ind^k)$ as a vector space, and is isomorphic to its opposite algebra as an algebra.
    \end{lemma}
    
    \begin{proof}
    The induction and restriction functors form a cyclic biadjoint pair for a finite group $G$ and one of its subgroups $H$, \textit{i.e.}, see \cite[Section 3.2]{MR3205569}. This relationship allows one to deduce the structure of $\End_n(\Res^{k})$ as a vector space and as an algebra.
    
    Alternatively in terms of bimodule homomorphisms, the natural transformations between $\Ind$ correspond to bimodule maps of multiplication on the right, and the ones for $\Res$ correspond to bimodule maps of multiplication on the left. Hence one obtains the same vector space, and the algebra multiplication reverses the order.
    \end{proof}

   Although we can provide a nice description for the vector space structure of $\End_n(\Ind^k)$ and $\End_n(\Res^k)$, the picture changes slightly for $\End_n(\Ind^k\Res^l)$ when $l\neq 0$ since in this case, a vector space basis is given by the orbit sums of the conjugation action of $\A_n$ on the tensor product $\GA{n+k-l}\otimes_{n-l}\GA{n}$, which is not a group algebra. So even though one can give an explicit description of orbit representatives, it is difficult to express the orbit sums explicitly. 
   
   \begin{proposition}
   \label{prop:vs-basis-general}
   A vector space basis of $\End_n(\Ind^k\Res^l)$ is given by
   \[
   \left\{ a^h \otimes_{n-l} b {\beta}_I^h: a\in \A_{n+k-l}, \: 
   b\in \widehat{\A}_{n-l} \widehat{\A}_{n-l+1} \cdots \widehat{A}_{n-1}, \: 
   {\beta}_I \in \mathcal{B}_{n-l,l}\cup \{ \id \} \right\}_{h\in \A_n},
   \] where $a^h := hah^{-1}$.
   \end{proposition}
   
   \begin{proof}
   We start by noting that an element of $\GA{n+k-l}\otimes_{n-l}\GA{n}$ can be written as $\displaystyle{\sum_{i\in I}}c_ia_i \otimes_{n-l} b_i$ where $c_i\in \mathbb{C}$, $a_i\in \A_{n+k-l}$ and $b_i$ is a representative from the quotient $\A_{n-l}\backslash \A_n$ after pushing all possible terms to the left of the tensor. Therefore $\{a_i\otimes_{n-l}b_i\}$ is a basis of $\GA{n+k-l}\otimes_{n-l}\GA{n}$. We will choose $b_i$ from the set of representatives in the disjoint union
   \[ \displaystyle \coprod_{x\in \mathcal{B}_{n-l,l} \cup \{ \id \}} 
   \widehat{\A}_{n-l} \widehat{\A}_{n-l+1} \widehat{\A}_{n-l+2}\cdots \widehat{\A}_{n-1}x,
   \] 
   given in Proposition~\ref{prop:rightcoset-rep-gen}.

   Let us look at the case $l=1$ first. In this case the representatives are $\widehat{\A}_{n-1} \sqcup \widehat{\A}_{n-1}{\beta}_n $.
   Given a basis element $a \otimes_{n-1} b$, $b$ is either in $\widehat{\A}_{n-1}$ or in $\widehat{\A}_{n-1}{\beta}_n$.
   
   For $h\in \A_{n-1}$, $a\in \A_{n+k-1}$ and $b\in \widehat{\A}_{n-1}$, we get 
   \begin{align*}
       h\cdot ( a \otimes_{n-1} b) &=  h(a \otimes_{n-1} b) h^{-1} \\
       &=  ha \otimes_{n-1} h^{-1} b \quad \text{ since $\widehat{\A}_{n-1}$ commutes with $\A_{n-1}$}\\
       &=  hah^{-1} \otimes_{n-1} b 
   \end{align*}
   and 
    \begin{align*}
       h\cdot ( a \otimes_{n-l} b{\beta}_n) &=  h(a \otimes_{n-l} b{\beta}_n) h^{-1} \\
       &=  hah^{-1} \otimes_{n-l} h b {\beta}_n h^{-1} \\
       &=  hah^{-1} \otimes_{n-l} b h{\beta}_n h^{-1} \\ 
       &= a^h \otimes_{n-1} b {\beta}_n^h. 
   \end{align*}
   
   More generally for any $l$, if $b\in \widehat{\A}_{n-l} \widehat{\A}
_{n-l+1} \widehat{\A}_{n-l+2}\cdots \widehat{\A}_{n-1}$, then 
   \begin{align*}
       h\cdot ( a \otimes_{n-l} b) &=  h(a \otimes_{n-l} b) h^{-1} \\
       &=  ha \otimes_{n-l} h^{-1} b \quad \text{ since $\widehat{\A}_{n-l} \widehat{\A}
_{n-l+1} \widehat{\A}_{n-l+2}\cdots \widehat{\A}_{n-1}$ } \\ 
& \hspace{2.65cm} \quad \text{commutes with $\A_{n-1}$}\\
       &=  hah^{-1} \otimes_{n-l} b \\
       &= a^h \otimes_{n-l} b
   \end{align*}
   and  
    \begin{align*}
       h\cdot ( a \otimes_{n-l} b{\beta}_I) &=  h(a \otimes_{n-l} b{\beta}_I) h^{-1} \\
       &=  hah^{-1} \otimes_{n-l} h b{\beta}_I h^{-1} \\
       &= hah^{-1} \otimes_{n-l} b h {\beta}_I h^{-1} \quad  \text{ since $\widehat{\A}_{n-l} \widehat{\A}
_{n-l+1} \widehat{\A}_{n-l+2}\cdots \widehat{\A}_{n-1}$ }  \\  
& \hspace{3.75cm}\quad  \text{commutes with $\A_{n-1}$} \\
       &=  hah^{-1} \otimes_{n-l} b h{\beta}_I h^{-1} \\ 
       &= a^h \otimes_{n-l} b {\beta}_I^h,  
   \end{align*}
    where $a^h = hah^{-1}$. 
\end{proof}

We do not have an explicit description for $\Hom_n(\Ind^{k_1}\Res^{l_1},\Ind^{k_2}\Res^{l_2})$, where $k_1-l_1=k_2-l_2$ but $k_1\not= k_2$. 
   
   In the next section, we will describe the algebra structure of the endomorphism spaces.

\subsection{Algebra structure of morphism spaces}
    
    In \cite{Ver06}, the author develops a general approach for obtaining a nice set of multiplicative generators for the Gelfand--Tsetlin subalgebra $\GZ{n} = \left\langle\zeta(\mathbb{C}\A_1),\ldots, \zeta(\mathbb{C}\A_n)\right\rangle$. His approach does not apply directly to our tower of groups, since we are in a non-multiplicity free setting. 
    


    The set $\End_n(\Ind_n^{n+1})$ is closed under precomposing with elements of the center $\End_n(\Id)\simeq Z(\GA{n})$ of the group algebra. Therefore the algebra $\End_n(\Ind_n^{n+1})$ can be decomposed as $\End_n(\Id)\otimes \D_{n,n+1}$, where conjugacy class sums of $\A_n$ form a basis of $\End_n(\Id)$ and $\widehat{\A}_n\cup \{\orbitsum_n({\beta}_{n+1})\}$ is a spanning set for the algebra $\D_{n,n+1}$. 
    
    Since the group $\widehat{\A}_n$ is generated by $\widehat{B} = \left\{\widehat{{\beta}}_1,\widehat{{\beta}}_2,\dots, \widehat{{\beta}}_{n}\right\}$, 
    a generating set for the algebra $\D_{n,n+1}$ is given by $\widehat{B}\cup \{\orbitsum_n({\beta}_{n+1})\}$. 
    The relations among the generators in $\widehat{B}$ is described in Theorem~\ref{thm:OOR-presentation}.  $\orbitsum_n({\beta}_{n+1})$ is a central element of $\mathbb{C}[\A_{n+1}]$ by Remark~\ref{rmk:orbitsum_beta_n_central}, therefore it commutes with the generators in $\widehat{B}$, \textit{i.e.}, $[\orbitsum_n({\beta}_{n+1}),\widehat{{\beta}}_i]=0$. However this is not enough to give a generators and relations description of the algebra $\D_{n,n+1}$. In particular, the only relation missing in order to give a complete generators and relations description $\D_{n,n+1}$ is the expression of $\orbitsum_n({\beta}_{n+1})^k$ in terms of the generators. The only nice pattern we have observed is $\orbitsum_n({\beta}_{n+1})^k = 2^{k-1}\orbitsum_n({\beta}_{n+1})$ for $n=1$ and $k$ odd. In fact, closed formulas are very difficult to find for all the other cases, \textit{i.e.}, $n=1$ and $k$ even, and $n>1$. However working in a concrete setting has the advantage of defining $\D_{n,n+1}$ as the subalgebra of $\GA{n+1}$ generated by $\widehat{B}\cup \{\orbitsum_n({\beta}_{n+1})\}$.
    
    Thus to determine the algebra structure of $\End_n(\Ind)$, it remains to describe the algebra structure of $\End_n(\Id)\simeq Z(\GA{n})$. It is difficult to express the structure coefficients if one chooses to work with the conjugacy class sums as a basis. Khonvanov's Heisenberg category provides a way to express centers of symmetric group algebras with generators and relations using a different basis $\{c_i\}_{i\in \mathbb{N}}$ defined via diagrammatics as certain bubbles. We lack a diagrammatic description for our category. We leave the algebra structure of $\End_n(\Id)\simeq Z(\GA{n})$ as an open problem.

    
    Next, we will write down the generating set for the algebra structure of the endomorphism algebra  $\End_n(\Ind^2)$. 
    Similar as before, $\End_n(\Ind^2)$ decomposes as $\End_n(\Id)\otimes \D_{n,n+2}$, where $\widehat{\A}_n\cup\widehat{\A}_{n+1}\cup \{\orbitsum_n({\beta}_{n+1})\} \cup \{\orbitsum_n({\beta}_{n+2})\} \cup \{\orbitsum_n({\beta}_{n+2}{\beta}_{n+1})\}$ span the algebra $\D_{n,n+2}$. 
    Since $\widehat{\A}_n$ is generated by   the swaps $\widehat{B} = \left\{\widehat{{\beta}}_1,\widehat{{\beta}}_2,\dots, \widehat{{\beta}}_{n}\right\}$ and $\widehat{\A}_{n+1}$ is generated by $B^{(1)} = \left\{ {\beta}_1^{(1)},{\beta}_2^{(1)},\ldots, {\beta}_{n+1}^{(1)}\right\}$, where the $\widehat{{\beta}}_i$ are the leftmost swaps of $\widehat{\A}_n$ and the ${\beta}_i^{(1)}$ are the leftmost swaps in each row in $\widehat{\A}_{n+1}$, respectively, a generating set for the algebra $\D_{n,n+2}$ is given by 
    \[
    \widehat{B} \cup B^{(1)}\cup \left\{\orbitsum_n({\beta}_{n+1}), \orbitsum_n({\beta}_{n+2}), \orbitsum_n({\beta}_{n+2}{\beta}_{n+1}) \right\}.
    \] 
    
    \input{fig5_1}

    More generally, consider $\A_n$ acting by conjugation on the group algebra $\GA{n+k+1}$. Then $\End_n(\Ind^k)$ decomposes as $\End_n(\Id)\otimes \D_{n,n+k+1}$, where the disjoint union  
    \[ 
    \coprod_{\ell=0}^{k} \widehat{\A}_{n+\ell} 
    \:\: 
    \sqcup
    \:\: 
    \coprod_{n<j_1<\ldots < j_s\leq n+k+1} \left\{\orbitsum_n({\beta}_{j_s}\cdots {\beta}_{j_1})\right\}. 
    \] 
    span the algebra $\D_{n,n+k+1}$. The group $\widehat{\A}_{n+\ell}$ is generated by 
    \[ 
    B^{(\ell)}=\{{\beta}_1^{(\ell)},\ldots, {\beta}_{n+\ell}^{(\ell)}\},
    \] 
    where the ${\beta}_i^{(\ell)}$ are the leftmost swaps in each row in $\widehat{\A}_{n+\ell}$, $0\leq \ell \leq k$. 
    This leads us to the following theorem: 
    \begin{theorem}
    \label{theorem:alg-gens}
    The endomorphism algebra $\End_n(\Ind^k)$ decomposes as 
    \[ 
    \End_n(\Id)\otimes \D_{n,n+k+1},
    \]
    where a generating set for the algebra $\D_{n,n+k+1}$ 
    is the disjoint union 
    \[ 
    \coprod_{\ell=0}^{k} B^{(\ell)} \:\:\sqcup\:\:
    \coprod_{n<j_1<\ldots < j_s\leq n+k+1} \left\{\orbitsum_n({\beta}_{j_s}\cdots {\beta}_{j_1})\right\}, 
    \] 
    and $B^{(\ell)} = \left\{{\beta}_1^{(\ell)},\ldots, {\beta}_{n+\ell}^{(\ell)}\right\}$, $0\leq \ell\leq k$.
    \end{theorem}

    As mentioned above, since the relations among the algebra generators are very complicated, we will refrain from explicitly writing them, but leave this as an open problem. 
    
\section{Future directions}
\label{section:future-direction}
    During our investigation of bimodules for the iterated wreath products of cyclic groups of order 2, we have encountered many questions, some of which remain unanswered. To list a few,
    
    \begin{itemize}
        \item A complete explicit description of all morphism spaces of $\mathcal{C}_n$ is unknown. Namely, the spaces $\Hom_n(\Ind^{k_1}\Res^{l_1},\Ind^{k_2}\Res^{l_2})$ where $k_1-k_2=l_1-l_2$ but $k_1\neq k_2$.
        
        \item A presentation of the algebras $\D_{n,n+k}$ appearing in the morphism spaces is unknown. We just have a set of generators and some of the defining relations.
        
        \item Young--Jucys--Murphy's elements play an important role in the Heisenberg categories. An open question is the analogues of these elements for iterated wreath products.
        
        \item An explicit description for the category of bimodules of the symmetric groups is a difficult question, however they can be recovered as functorial images of Khovanov's Heisenberg category. Existence of such an abstract category for iterated wreath products is an open question. One important difference is that the Mackey formula for the symmetric group $S_n$ is independent of $n$, while this formula depends on $n$ for iterated wreath product $\A_n$. 
        
        \item Heisenberg categories admit diagrammatic interpretations which benefit from the diagrammatics of the symmetric groups. Whether such a description exists for the category $\mathcal{C}_n$ is an open question.
        
        \item It would be interesting to take iterated wreath products of other cyclic groups or other symmetric groups and compare how their bimodule categories differ from the $S_2$ case. 
        
        \item We do not know the Grothendieck groups or traces of the categories $\mathcal{C}_n$.
        
    \end{itemize}

\subsection*{Acknowledgements} 
    The first author would like to acknowledge the Department of Mathematics at Galatasaray University in Turkey for excellent working environment.

\subsection*{Author Contributions} Im and O\u{g}uz contributed equally in the research-aspects and the writing of this article. 

\subsection*{Funding Information} Im was partially supported by National Academy of Sciences at Washington D.C. and Turkish Scientific and Technological Research Council TUBITAK 115F412. 


\subsection*{Conflict of Interest}
The authors have no conflicts of interest to declare that are relevant to the content of this article.

\subsection*{Data Availability Statement}
Data sharing is not applicable to this article as no data sets were generated or analyzed during the current study.

\subsection*{Supplementary Materials} Not applicable. 

\bibliography{cat}   

\end{document}

%% file: fig2_1.tex
\begin{figure}
    \centering
\begin{tikzpicture}[scale=0.6]
\draw[thick] (0.1,0.1) -- (2,2);    
\draw[thick] (-.1,0.1) -- (-2,2);
\draw[thick] (2,2) -- (3.55,3.8);
\draw[thick] (2,2) -- (0.45,3.8);
\draw[thick] (-2,2) -- (-3.55,3.8);
\draw[thick] (-2,2) -- (-0.45,3.8);
\draw[thick] (0.15,0) arc (0:360:1.5mm); 
\draw[thick,fill] (2.15,2) arc (0:360:1.5mm); 
\draw[thick,fill] (-1.85,2) arc (0:360:1.5mm); 

\draw[thick,fill] (3.65,3.8) arc (0:360:1.5mm); 
\draw[thick,fill] (-3.4,3.8) arc (0:360:1.5mm); 
\draw[thick,fill] (0.6,3.8) arc (0:360:1.5mm); 
\draw[thick,fill] (-.3,3.8) arc (0:360:1.5mm); 

\node at (-3.55,4.3) {\Large $1$};
\node at (-.47,4.3) {\Large $2$};
\node at (.45,4.3) {\Large $3$};
\node at (3.5,4.3) {\Large $4$};

\draw[thick,blue] (4,3.6) arc (0:360:2);
\end{tikzpicture}
    \caption{The embedding of $\A_1\stackrel{i}{\longhookrightarrow} \A_2$ fixes the circled branch. }
    \label{fig2_1}
\end{figure}

%% file: fig3_1.tex
\begin{figure}
    \centering
\begin{tikzpicture}[scale=0.6]
\draw[thick] (0.1,0.1) -- (2,2);    
\draw[thick] (-0.1,0.1) -- (-2,2);

\draw[thick] (-2,2) -- (-3.55,3.8);
\draw[thick] (-2,2) -- (-0.45,3.8);
\draw[thick] (0.15,0) arc (0:360:1.5mm);
\draw[thick,fill] (2.15,2) arc (0:360:1.5mm);
\draw[thick,fill] (-1.85,2) arc (0:360:1.5mm);

\draw[thick,fill] (-3.4,3.8) arc (0:360:1.5mm);
 
\draw[thick,fill] (-.3,3.8) arc (0:360:1.5mm);

\draw[thick,blue] (4.7,3.6) arc (0:360:2);
\draw[thick,blue] (.6,4.8) arc (0:360:1.35);
\draw[thick,blue] (-2.7,4.8) arc (0:360:1.35);

\node at (2.8,3.6) {\Large $\widehat{\A}_{n+1}$};
\node at (-4.1,4.9) {\Large $\A_n$};
\node at (-.6,4.9) {\Large $\widehat{\A}_n$};

\end{tikzpicture}
    \caption{Images of $\A_n$, $\widehat{\A}_n$ and $\widehat{\A}_{n+1}$ represented in the tree diagram of $T_{n+2}$.}
    \label{fig3_1}
\end{figure}
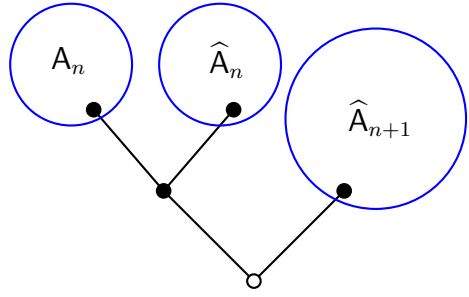

%% file: fig5_1.tex
\begin{figure}
    \centering
\begin{tikzpicture}[scale=0.68]
\draw[thick] (0,0) -- (8,8);
\draw[thick] (0,0) -- (-8,8);

\draw[thick] (4,4) -- (0.5,8);
\draw[thick] (-4,4) -- (-0.5,8);

\draw[thick] (6,6) -- (4.5,8);
\draw[thick] (2.2,6) -- (3.5,8);
\draw[thick] (-6,6) -- (-4.5,8);
\draw[thick] (-2.25,6) -- (-3.55,8);

\draw[thick] (-7,7) -- (-6.5,8);
\draw[thick] (-5.25,7) -- (-5.75,8);

\draw[thick] (-2.85,7) -- (-2.35,8);
\draw[thick] (-1.35,7) -- (-1.85,8);

\draw[thick] (7,7) -- (6.5,8);
\draw[thick] (5.25,7) -- (5.75,8);

\draw[thick] (2.85,7) -- (2.35,8);
\draw[thick] (1.35,7) -- (1.85,8);

\node at (-7.05,8.15) {\Large $\beta_1$};
\node at (-6,7.15) {\Large $\beta_2$};
\node at (-4,5.5  ) {\Large $\beta_3$};
\node at ( 0,2.25) {\Large $\beta_4$};

\node at (-2.9,8.25) {\Large $\widehat{\beta}_1$};
\node at (-2.1,7.15) {\Large $\widehat{\beta}_2$};

\node at ( 1.25,8.2) {\Large $\beta_1^{(1)}$};
\node at ( 2.05,7.18) {\Large $\beta_2^{(1)}$};
\node at ( 4.15,5.5 ) {\Large $\beta_3^{(1)}$};

\draw[thick] (-.9,1.1) .. controls (-.5,2) and (.5,2) .. (.9,1.1);
\draw[thick] (-.9,1.1) -- (-1.0,1.3);
\draw[thick] (-.9,1.1) -- (-.7,1.1);
\draw[thick] (.9,1.1) -- (1.0,1.3);
\draw[thick] (.9,1.1) -- (.7,1.1);
\draw[thick] (-4.6,4.75) .. controls (-4.35,5.25) and (-3.85,5.25) .. (-3.6,4.75);
\draw[thick] (-4.6,4.75) -- (-4.6,5.05);
\draw[thick] (-4.6,4.75) -- (-4.4,4.75);
\draw[thick] (-3.6,4.75) -- (-3.6,5.05);
\draw[thick] (-3.6,4.75) -- (-3.8,4.75);
\draw[thick] (4.55,4.65) .. controls (4.3,5.15) and (3.8,5.15) .. (3.55,4.65);
\draw[thick] (4.55,4.65) -- (4.55,4.95);
\draw[thick] (4.55,4.65) -- (4.35,4.65);
\draw[thick] (3.55,4.65) -- (3.55,4.95);
\draw[thick] (3.55,4.65) -- (3.75,4.65);

\draw[thick] (-6.3,6.5) .. controls (-6.2,6.75) and (-5.9,6.75) .. (-5.8,6.5);
\draw[thick] (-6.3,6.5) -- (-6.35,6.7);
\draw[thick] (-6.3,6.5) -- (-6.15,6.5);
\draw[thick] (-5.8,6.5) -- (-5.75,6.7);
\draw[thick] (-5.8,6.5) -- (-5.95,6.5);
\draw[thick] (-2.45,6.5) .. controls (-2.35,6.75) and (-2.05,6.75) .. (-1.95,6.5);
\draw[thick] (-2.45,6.5) -- (-2.5,6.7);
\draw[thick] (-2.45,6.5) -- (-2.3,6.5);
\draw[thick] (-1.95,6.5) -- (-1.9,6.7);
\draw[thick] (-1.95,6.5) -- (-2.1,6.5);
\draw[thick] (2.4,6.5) .. controls (2.3,6.75) and (2.,6.75) .. (1.9,6.5);
\draw[thick] (2.4,6.5) -- (2.45,6.7);
\draw[thick] (2.4,6.5) -- (2.25,6.5);
\draw[thick] (1.9,6.5) -- (1.85,6.7);
\draw[thick] (1.9,6.5) -- (2.05,6.5);

\draw[thick] (-7.35,7.5) .. controls (-7.25,7.75) and (-7.0,7.75) .. (-6.9,7.5);
\draw[thick] (-7.35,7.5) -- (-7.35,7.75);
\draw[thick] (-7.35,7.5) -- (-7.15,7.5);
\draw[thick] (-6.85,7.5) -- (-6.85,7.75);
\draw[thick] (-6.85,7.5) -- (-7.0,7.5);
\draw[thick] (-3.15,7.5) .. controls (-3.05,7.75) and (-2.8,7.75) .. (-2.7,7.5);
\draw[thick] (-3.15,7.5) -- (-3.15,7.75);
\draw[thick] (-3.15,7.5) -- (-3.0,7.5);
\draw[thick] (-2.7,7.5) -- (-2.7,7.75);
\draw[thick] (-2.7,7.5) -- (-2.85,7.5);
\draw[thick] (1.5,7.5) .. controls (1.4,7.75) and (1.1,7.75) .. (1,7.5);
\draw[thick] (1.5,7.5) -- (1.5,7.75);
\draw[thick] (1.5,7.5) -- (1.35,7.5);
\draw[thick] (1,7.5) -- (1,7.75);
\draw[thick] (1,7.5) -- (1.15,7.5);

\node at (-8,2) {\Large $\A_4$};

\draw[thick,blue] (-9,3.5) rectangle (-0.2,10);
\node at (-7,4.5) {\Large ${\color{blue}\A_3}$};

\draw[thick,red] (-8.5,5.95) rectangle (-4,8.8);
\node at (-7.75,6.5) {\Large ${\color{red}\A_2}$};

\draw[thick,red] (-3.7,5.95) rectangle (-.4,8.8);
\node at (-1,6.5) {\Large ${\color{red}\widehat{\A}_2}$};

\draw[thick,magenta] (.4,8.8) rectangle (3.7,5.95);
\node at (1,6.5) {\Large ${\color{magenta}\A_2'}$};

\draw[thick,magenta] (4,8.8) rectangle (8.5,5.95);
\node at (7.75,6.5) {\Large ${\color{magenta}\widehat{\A}_2'}$};

\draw[thick,blue] (0.2,10) rectangle (9,3.5);
\node at (7,4.5) {\Large ${\color{blue}\widehat{\A}_3}$};

\end{tikzpicture}
    \caption{The algebra $\End_2(\Ind^2)$ decomposes as $\End_2(\Id)$ $\otimes$ $\D_{2,4}$. A generating set for 
    $\D_{2,4}$ 
    is $\widehat{B} \cup B^{(1)}\cup \{\orbitsum_2(\beta_{3}), \orbitsum_2(\beta_{4}), \orbitsum_2(\beta_{4}\beta_{3}) \}$, where $\widehat{B} = \{\widehat{\beta}_1,\widehat{\beta}_2 \}$ and $B^{(1)}=\{\beta_1^{(1)},\beta_2^{(1)},\beta_3^{(1)}\}$. 
    }
    \label{fig5_1}
\end{figure}